\documentclass[a4paper,10pt]{article}

\pdfoutput=1

\usepackage[round, sort, numbers, authoryear]{natbib}
\usepackage{amsmath}
\usepackage{dsfont}
\usepackage{amsthm}	
\usepackage{latexsym}
\usepackage[psamsfonts]{amssymb}
\usepackage[T1]{fontenc}
\usepackage{xcolor}
\usepackage{enumitem}

\newcommand{\beginsupplement}{%
	\setcounter{table}{0}
	\renewcommand{\thetable}{S\arabic{table}}%
	\setcounter{figure}{0}
	\renewcommand{\thefigure}{S\arabic{figure}}%
	\renewcommand{\appendixname}{Supplementary material}
	\appendix
	\renewcommand{\thesection}{S\arabic{section}}
	\setcounter{section}{0}
	\renewcommand{\thesubsection}{S\arabic{subsection}}
	\setcounter{subsection}{0}
	\renewcommand{\theequation}{S\arabic{equation}}
	\setcounter{equation}{0}
	\setcounter{page}{1}
}

\usepackage{graphicx}
{\catcode `\@=11 \global\let\AddToReset=\@addtoreset}
\AddToReset{equation}{section}
\renewcommand{\theequation}{\thesection.\arabic{equation}}

\newtheorem{Theorem}{Theorem}[section]
\newtheorem{Lemma}{\bf Lemma}[section]
\newtheorem{Corollary}{\bf Corollary}[section]

\newtheorem{Proposition}{\bf Proposition}[section]
\newtheorem{@definition}{\sc Definition}[section]
\newenvironment{definition}{\begin{@definition}\rm}{\end{@definition}}
\newtheorem{@remark}{\sc Remark}[section]
\newenvironment{remark}{\begin{@remark}\rm}{\end{@remark}}
\newtheorem{@example}{\sc Example}[section]

\def\eqd{\stackrel{\mbox{$\scriptstyle d$}}{=}}
\def\convd{\stackrel{\mbox{$\scriptstyle d$}}{\to}}
\def\convP{\stackrel{\mbox{$\scriptstyle P$}}{\to}}
\def\dd{~\mathrm d}
\def\E{\mathbb{E}}

\def\e{\mathrm e}

\def\R{\mathbb R}
\def\N{\mathcal{N}}
\def\VG{\mathrm{VG}}
\def\Var{\mathrm{Var}}
\def\Cov{\mathrm{Cov}}
\def\Corr{\mathrm{Corr}}

\begin{document}
	\title{Asymptotic normality in linear regression with approximately sparse structure}
	
	\author{Saulius Jokubaitis$^{1}$, Remigijus Leipus$^{1}$\thanks{Supported by grant
			No.\ S-MIP-20-16 from the Research Council of Lithuania} }
	\date{\today
		\\  \small
		\vskip.2cm
		\begin{itemize}
			\item [$^1$]\hskip -.15cm Institute of Applied Mathematics, Faculty of Mathematics and Informatics,\\
			\hskip -.15cm Vilnius University, Naugarduko 24, Vilnius LT-03225, Lithuania
		\end{itemize}
	}
	
	\maketitle
	
	\begin{abstract}
		In this paper we study the asymptotic normality in high-dimensional linear regression. We focus on the case where the covariance matrix of the regression variables has a KMS structure, in asymptotic settings where the number of predictors, $p$, is proportional to the number of observations, $n$. The main result of the paper is the derivation of the exact asymptotic distribution for the suitably centered and normalized squared norm of the product between predictor matrix, $\mathbb{X}$, and outcome variable, $Y$, i.e. the statistic $\|\mathbb{X}'Y\|_{2}^{2}$.
		Additionally, we consider a specific case of approximate sparsity of the model parameter vector $\beta$ and perform a Monte-Carlo simulation study. The simulation results suggest that the statistic approaches the limiting distribution fairly quickly even under high variable multi-correlation and relatively small number of observations, suggesting possible applications to the construction of statistical testing procedures for the real-world data and related problems.
		
		\medskip
		
		{\small
			
			\noindent {\em MSC}: 60F05, 62E20, 62J99
			
			\noindent {\em Keywords:} linear regression, sparsity, asymptotic normality, variance-gamma distribution
		}
		
	\end{abstract}
	
\section{Introduction}

\noindent Consider a linear regression model
\begin{align}\label{model}
	Y = \mathbb{X} \beta + \varepsilon,
\end{align}
where $Y:= (y_1, \ldots, y_n)' \in \mathbb{R}^{n \times 1}$ are $n$ observations of outcome and $\mathbb{X}~=~(X_1, \ldots, X_n)' \in \mathbb{R}^{n \times p}$ are $p$-dimensional predictors with $X_1, \ldots, X_n$ being i.i.d.\ $p \times 1$ random vectors \mbox{$X_i=(X_{1,i},\dots,X_{p,i})'$,} which are normally distributed with zero mean and covariance matrix $\Sigma$, denoted $X_{i} \stackrel{d}{=} \mathcal{N}_p(0,\Sigma)$. We assume that the covariance matrix $\Sigma$ has a form
\begin{eqnarray} \label{KMS}
	\Sigma &=& (\rho^{|i-j|})_{i,j=1}^p \ =\
	\left[
	\begin{array}{cccc}
		1 & \rho  & \dots & \rho^{p-1} \\
		\rho & 1 & \dots & \rho^{p-2} \\
		\vdots & \vdots & \ddots & \vdots \\
		\rho^{p-1} & \rho^{p-2} & \dots & 1 \\
	\end{array}
	\right],
\end{eqnarray}
if $0 < |\rho|<1$, and $\Sigma = I_p$ if $\rho = 0$ (here and below $I_p$ denotes the $p \times p$ identity matrix). This matrix is often called
the Kac–Murdock–Szeg\"o (KMS) matrix, originally introduced in \cite{kms1953}. As the
autocorrelation matrix of corresponding causal AR(1) processes, KMS matrix is positive definite
and is considered due to the wide array of applications in the literature and its' well known spectral properties (see, e.g., \cite{FIKIORIS2018182} for a thorough literature review). Further, $\varepsilon := (\varepsilon_1, \ldots, \varepsilon_n)' \in \mathbb{R}^{n \times 1} \stackrel{d}{=} \mathcal{N}(0, \sigma^2_{\varepsilon} I_n)$ are unobserved i.i.d.\ errors with $\mathbb{E}\varepsilon_i= 0$, $\operatorname{Var}(\varepsilon_i) = \sigma_{\varepsilon}^{2} > 0$, and $\beta :=(\beta_1, \ldots, \beta_p)' \in \mathbb{R}^{p \times 1}$ is an unknown $p$-dimensional parameter. In practice, the assumption that $\mathbb{E}X_i = 0$ can be untenable and it may be appropriate to add an intercept to the linear model (\ref{model}), however, for simplicity, throughout this paper we will assume that the intercept is known and the variables are centered.

This paper is concerned with the derivation of the exact asymptotic distribution for the suitably centered and normalized squared norm $\|\mathbb{X}'Y\|_{2}^{2}$ under the assumption of the KMS type covariance structure in \eqref{KMS}, where $p$ and $n$ are assumed large.
Throughout the paper we assume that $p, n \to \infty$ and $p/n \to c \in (0, \infty)$. We are particularly interested in cases where $p >n$. Statistics of such form arise in various applications in the context of high-dimensional linear regression, and under normality assumptions, general results can be derived using random matrix theory through Wishart distributions (see, e.g., \cite{dicker2014}). In our paper we approach the problem following an observation by \cite{phdthesis_gaunt}, that the distribution of product of Gaussian random variables admits a variance-gamma distribution, resulting in a set of attractive properties.  In addition to the $\ell_2$-norm statistic, we find that the obtained results can be easily extended towards alternative forms of the statistic, e.g., by using a different norm, which would reduce the problem to manipulating variance-gamma distribution, thus suggesting possible further research cases and useful extensions.

Additionally, we examine a specific case of parameter $\beta$ by considering $\beta_j = j^{-1}$, $j\ge 1$. Similar structures of the vector $\beta$ are often found in the literature when  approximate sparsity of the  coefficients in the linear regression model (\ref{model}) is assumed. See, e.g.,  \cite{ing2019model} and \cite{cha2021inference} for a broader view towards sparsity requirements and its' implications to specific high-dimensional algorithms; \cite{shibata1980asymptotically} and \cite{ing2007accumulated} for model selection problems in autoregressive time series models; or \cite{Chernozhukov2013},  \cite{javanmard2014confidence},  \cite{zhang2014confidence}, \cite{caner2018asymptotically}, \cite{belloni2018high}, \cite{gold2020inference}, \cite{ning2020doubly}, \cite{guo2021doubly} for applications on inference of high-dimensional models and high-dimensional instrumental variable (IV) regression models. Performing Monte Carlo simulations, we find that the empirical distributions of the corresponding statistic approach the limiting distribution reasonably quickly even for large values of $\rho$ and $c$, suggesting that the assumption of sparse structure can be included in the applications and statistical tests.

In this paper, $\eqd$, $\convd$ and $\convP$ denote the equality of distributions, convergence of distributions and convergence in probability, respectively. $C$ stands for a generic positive constant which may assume different values at various locations.
${\bf 1}_{A}$ denotes the indicator function of a set $A$.

The structure of the paper is as follows. In Section \ref{sec:main} we present the main results of the paper. In Section~\ref{sec:properties_vg} we present useful properties of variance-gamma distribution, which are used in Section~\ref{sec:auxiliary} in order to prove some auxiliary results.  In Section \ref{sec:proofs} we present the proof of the main result. Finally, in Section \ref{sec:sparsity} we provide an example of the main result under imposed approximate sparsity assumption for the parameter $\beta$ of the model (\ref{model}). Technical results are presented in Appendix A, while, for brevity, some straightforward yet tedious proofs are presented in the Supplementary material.

\section{Main results}	\label{sec:main}

\noindent In this section we formulate the main results on the normality of statistic $\|\mathbb{X}'Y\|_2^2$.
Introduce the notations:
\begin{eqnarray}
	\kappa_{1,p} &:=& \sum_{k=1}^p \sum_{l=1}^{p}\beta_k \beta_{l} \rho^{|k-l|}, \label{eq:kappa1} \\
	\kappa_{2,p} &:=& \sum_{k=1}^p \Big(\sum_{l=1}^p \beta_l \rho^{|k-l|}\Big)^2, \label{eq:kappa2} \\
	\kappa_{3,p} &:=& \sum_{k,l,j,j'=1}^p \beta_{j} \beta_{j'} \rho^{|k-j|}
	\rho^{|l-j'|} \rho^{|k-l|}.\label{eq:kappa3}
\end{eqnarray}
It is easy to see that, under $\sum_{j=1}^\infty \beta_j^2<\infty$, there exist limits
\begin{eqnarray*}
	\kappa_i &=&\lim_{p\to\infty} \kappa_{i,p}, \ i=1,2,3.
\end{eqnarray*}
Obviously, $\kappa_{2,p} \geq 0$. Moreover, since $(\rho^{|i-j|})_{i,j=1}^p$ is positive semi-definite, $\kappa_{i,p} \geq 0$, $i = 1,3$. Indeed,
$\sum_{k,l=1}^{p} \rho^{|k-l|}a_{k} a_{l} \geq 0$,
thus it suffices to take $a_k = \beta_k$ for $i=1$ and $a_k = \sum_{j=1}^{p}\beta_j \rho^{|k-j|}$ for $i=3$.

Our first main result is the following theorem.

\begin{Theorem}\label{th:wider}
	Assume the model in \eqref{model} with covariance structure in \eqref{KMS}. Let
	$n\to\infty$ and let $p=p_n$ satisfy
	\begin{eqnarray} \label{pnas}
		p\to\infty,  && \frac{p}{n}\to c\in (0,\infty).
	\end{eqnarray}
	Let also the $\beta_j$ satisfy
	\begin{eqnarray}\label{bnas}
		\sum_{j=1}^\infty \beta_j^2&<&\infty.
	\end{eqnarray}
	Then
	\begin{eqnarray}\label{eq:th}
		\frac{\|\mathbb{X}'Y\|_2^2 - n^2\kappa_{2,p} -  p n (\kappa_{1,p} + \sigma_{\varepsilon}^2) }{n^{3/2}} &\convd& {\cal N}(0,s^2),
	\end{eqnarray}
	where variance $s^2$ has the structure
	\begin{eqnarray} \label{eq:th_main_sigma}
		s^2&=&	4 \kappa_{2}^2 + 4(\kappa_{1} + \sigma_{\varepsilon}^2)\left(2 \kappa_{2} c + \kappa_{3} \right) + 2c (\kappa_{1} + \sigma_{\varepsilon}^2)^2  \Big( c + \frac{1 + \rho^2}{ 1 - \rho^2 }\Big).~~~~~~
	\end{eqnarray}
\end{Theorem}

Our second main result deals with the case where the centering sequence in \eqref{eq:th} is modified to include the limiting values of $\kappa_{i,p}$, $i=1,2$.

\begin{Theorem}\label{th:main}
	Let the assumptions of Theorem \ref{th:wider} hold. In addition, assume that $\sum_{j=p+1}^{\infty}\beta_j^2 = o(p^{-1/2})$ and $\sup_{j \geq 1} |\beta_j| j^{\alpha} < \infty$ with  $\alpha > 1/2$.  Then,
	\begin{eqnarray}\label{eq:th_2}
		\frac{\|\mathbb{X}'Y\|_2^{2} - n^2\big(\kappa_{2} +  c (\kappa_{1} + \sigma_{\varepsilon}^2)\big)}{n^{3/2}} &\convd&  \mathcal{N}(0,s^2).
	\end{eqnarray}
\end{Theorem}

The proofs of these theorems are given in Section~\ref{sec:proofs}.

\begin{remark}
	For alternative expressions of $\kappa_1$, $\kappa_2$ and $\kappa_3$, see Lemma~\ref{lemma:alt_expressions} below.
\end{remark}

Define
\begin{eqnarray*}
	\beta(x) &:=& \sum_{j=1}^\infty \beta_j^2 x^j, \ \ |x|\le 1.
\end{eqnarray*}

The following corollary deals with the case when $\rho = 0$, i.e., $\Sigma = I_{p}$. The result easily follows from Theorem \ref{th:main}, noting that in this case $\kappa_{i} = \beta(1)$, $i = 1,2,3$.

\begin{Corollary} \label{cor:Ip}
	Assume a model \eqref{model} with covariance structure $\Sigma = I_p$.
	Let assumptions \eqref{pnas} and \eqref{bnas} be satisfied. In addition, assume that $\sum_{j=p+1}^{\infty}\beta_j^2 = o(p^{-1/2})$ and $\sup_{j \geq 1} |\beta_j| j^{\alpha} < \infty$ with  $\alpha > 1/2$. Then,
	\begin{eqnarray}
		\frac{\|\mathbb{X}'Y\|_2^{2} - n^2 (\beta(1)(1+c) + c \sigma_{\varepsilon}^2)}{n^{3/2}} &\convd& \mathcal{N}(0,s^2),
	\end{eqnarray}
	where
	\begin{eqnarray}
		s^2&=& 2\beta(1)^2 \big(4+5c+c^2 \big) + 4\beta(1)\sigma_{\varepsilon}^2\big(1+3c+c^2 \big) + 2\sigma_{\varepsilon}^4(c+c^2).~~~~~~
	\end{eqnarray}
	
\end{Corollary}

\section{Properties of the variance-gamma distribution}
\label{sec:properties_vg}

\noindent
In this section we provide some properties of the variance-gamma distribution, which will be used in the following proofs.

Recall that the variance-gamma distribution with parameters $r>0$, $\theta\in \R$, $\sigma>0$
and $\mu\in\R$ has density
\begin{eqnarray}\label{VG}
	f^{\rm VG}(x)=\frac{1}{\sigma\sqrt\pi\Gamma(r/2)}\,\e^{\theta(x-\mu)/\sigma^2}
	\bigg(\frac{|x-\mu|}{2\sqrt{\theta^2+\sigma^2}}\bigg)^{(r-1)/2}K_{(r-1)/2}
	\bigg(\frac{\sqrt{\theta^2+\sigma^2}}{\sigma^2}\, |x-\mu|\bigg),\nonumber\\
\end{eqnarray}
where $x\in \R$, $K_\nu(x)$ is the modified Bessel function of the second kind. For a random variable $Q$  with density \eqref{VG} we write $Q\stackrel{d}{=} \VG(r,\theta,\sigma,\mu)$. Let $\Gamma(a,b)$, $a>0$, $b>0$, denote the gamma distribution with density
\begin{eqnarray*}
	f^{\rm G}(x) &=& \frac{b^a}{\Gamma(a)}\, x^{a-1}\e^{-b x}, \ \ x>0.
\end{eqnarray*}
It holds that
\begin{eqnarray}\label{repr}
	Q &\eqd & \mu+\theta W_r + \sigma \sqrt{W_r} U,
\end{eqnarray}
where $W_r\stackrel{d}{=} \Gamma(r/2,1/2)$, $U\stackrel{d}{=} \N(0,1)$,
$W_r$ and $U$ are independent. The characteristic function of $Q\stackrel{d}{=} \VG(r,\theta,\sigma,\mu)$
has a form (see, e.g., \cite{RePEc:oup:revfin:v:2:y:1998:i:1:p:79-105.}, \cite{Kotz2001TheLD})
\begin{eqnarray}\label{chVG}
	\varphi_{Q}(t) &=& \frac{\e^{{\rm i}\mu t}}{(1+\sigma^2 t -2{\rm i}\theta t)^{r/2}}, \ \ t\in \R.
\end{eqnarray}

\smallskip

\noindent We note the following properties of the variance-gamma distribution.

\begin{enumerate}[label={(\rm \roman*)}]
	\item If $Q_1\stackrel{d}{=} \VG(r_1,\theta,\sigma,\mu_1)$ and $Q_2\stackrel{d}{=} \VG(r_2,\theta,\sigma,\mu_2)$ are independent random variables then
	\begin{eqnarray*}
		Q_1+Q_2 &\stackrel{d}{=}& \VG(r_1+r_2,\theta,\sigma,\mu_1+\mu_2).
	\end{eqnarray*}
	\item If $Q\stackrel{d}{=} \VG(r,\theta,\sigma,\mu)$, then for any $a>0$
	\begin{eqnarray*}
		a Q &\stackrel{d}{=}& \VG(r,a\theta,a\sigma,a\mu).
	\end{eqnarray*}
\end{enumerate}

\noindent The following proposition is crucial for our purposes.

\begin{Proposition} \label{pr:VG}
	{\rm (i)} If $(\xi_1,\xi_2)'\stackrel{d}{=} \N(0,\Sigma)$, where $\Sigma=\Big(
	\begin{array}{cc}
		\sigma_1^2 & \rho\sigma_1\sigma_2 \\
		\rho\sigma_1\sigma_2& \sigma_2^2
	\end{array}
	\Big)$,
	then
	\begin{eqnarray*}
		\xi_1\xi_2 &\stackrel{d}{=}& \VG(1,\rho\sigma_1\sigma_2, \sqrt{1-\rho^2}\sigma_1\sigma_2,0).
	\end{eqnarray*}
	
	\noindent
	{\rm (ii)} If $(\xi_{1j},\xi_{2j})'$, $j=1,\dots,n$ are i.i.d.\ random vectors
	with common distribution $\N(0,\Sigma)$, then
	\begin{eqnarray*}
		\sum_{j=1}^n  \xi_{1j}\xi_{2j} &\stackrel{d}{=}& \VG(n,\rho\sigma_1\sigma_2, \sqrt{1-\rho^2}\sigma_1\sigma_2,0)
	\end{eqnarray*}
	and
	\begin{eqnarray*}
		\sum_{j=1}^n  \xi_{1j}\xi_{2j}  &\eqd& \sigma_1\sigma_2(\rho W_n+\sqrt{1-\rho^2}\sqrt{W_n} U),
	\end{eqnarray*}
	where $W_n\stackrel{d}{=} \Gamma (n/2,1/2)$ and $U\stackrel{d}{=} \N(0,1)$ are independent random variables.
	
	\smallskip
	
	\noindent {\rm (iii)} Assume that $(\xi_{1j}^{(1)},\dots,\xi_{1j}^{(p)}, \xi_{2j})'$, $j=1,\dots,n$, are i.i.d.\ copies of $(\xi_1^{(1)},\dots,\xi_1^{(p)}, \xi_2)'\eqd \N(0,\Sigma^{(p)})$ and let
	$\rho^{(kl)}:=\Corr (\xi^{(k)}_1,\xi^{(l)}_1)$, $\rho^{(k)}:=
	\Corr(\xi_1^{(k)}, \xi_2)$, $(\sigma_1^{(k)})^2:=\Var(\xi_1^{(k)})$, $\sigma_2^2:=\Var(\xi_{2})$, $k,l=1,\dots,p$. Then
	\begin{eqnarray*}
		\left(
		\begin{array}{c}
			\sum_{j=1}^n  \xi^{(1)}_{1j}\xi_{2j}  \\
			\vdots \\
			\sum_{j=1}^n  \xi_{1j}^{(p)}\xi_{2j}
		\end{array}
		\right) &\eqd& \left(\begin{array}{c}
			\sigma_1^{(1)}\sigma_2(\rho^{(1)} W_n+\sqrt{1-(\rho^{(1)})^2}\sqrt{W_n} U_1) \\
			\vdots \\
			\sigma_1^{(p)}\sigma_2(\rho^{(p)} W_n+\sqrt{1-(\rho^{(p)})^2}\sqrt{W_n} U_p)
		\end{array}
		\right),
	\end{eqnarray*}
	where $(U_1,\dots,U_p)'\stackrel{d}{=} \N(0, \Sigma_U)$, $\Sigma_U=(\sigma_U^{(kl)})$ with
	\begin{eqnarray}\label{UkUl}
		\sigma_U^{(k,l)}&=&\E U_k U_l \ =\ \frac{\rho^{(kl)}-\rho^{(k)}\rho^{(l)}}
		{\sqrt{1-(\rho^{(k)})^2}\sqrt{1-(\rho^{(l)})^2}}, \ \ k,l=1,\dots,p.
	\end{eqnarray}
\end{Proposition}

\begin{proof} The statements in (i), (ii) are well-known, see e.g. \cite{RePEc:bla:stanee:v:73:y:2019:i:2:p:176-179}.
	The proof of part (iii) follows from Lemma \ref{VG-lem}.
\end{proof}
\begin{Lemma}\label{VG-lem}
	Assume that $(\xi_1^{(1)},\dots,\xi_1^{(p)}, \xi_2)'$ has distribution $\N(0,\Sigma^{(p)})$ and let $\rho^{(kl)}:=\Corr (\xi^{(k)}_1,\xi^{(l)}_1)$, $\rho^{(k)}:=
	\Corr(\xi_{1}^{(k)}, \xi_2)$, $(\sigma_1^{(k)})^2:=\Var(\xi_1^{(k)})$, $\sigma_2^2:=\Var(\xi_2)$, $k,l=1,\dots,p$. Then
	\begin{eqnarray*}
		\left(
		\begin{array}{c}
			\xi^{(1)}_1\xi_2  \\
			\vdots \\
			\xi_1^{(p)}\xi_2
		\end{array}
		\right) &\eqd& \left(\begin{array}{c}
			\sigma_1^{(1)}\sigma_2\big(\rho^{(1)} W_1+\sqrt{1-(\rho^{(1)})^2}\sqrt{W_1} U_1\big) \\
			\vdots \\
			\sigma_1^{(p)}\sigma_2\big(\rho^{(p)} W_1+\sqrt{1-(\rho^{(p)})^2}\sqrt{W_1} U_p\big)
		\end{array}
		\right),
	\end{eqnarray*}
	where $W_1\stackrel{d}{=} \Gamma(1/2,1/2)$, $(U_1,\dots,U_p)'$ is, independent of $W_1$, zero mean normal vector with covariances in \eqref{UkUl}.
\end{Lemma}

\noindent {\sc Proof.} It suffices to prove that for any $(t_1,\dots,t_p)\in \R^p$ it holds
\begin{eqnarray}\label{vd}
	\Big(\sum_{k=1}^p t_k \xi^{(k)}_1\Big)\xi_2 &\eqd& \sigma_2\sum_{k=1}^p t_k \sigma_1^{(k)} \big(\rho^{(k)} W_1 + \sqrt{1-(\rho^{(k)})^2}\sqrt{W_1}U_k\big).
\end{eqnarray}
Since
\begin{eqnarray*}
	\sum_{k=1}^p t_k \xi^{(k)}_1 &\stackrel{d}{=}& \N\Big(0, \sum_{k,l=1}^p t_k t_l \rho^{(kl)}\sigma_1^{(k)}\sigma_1^{(l)}\Big),\ \ \xi_2\ \stackrel{d}{=}\ \N(0,\sigma_2^2),
\end{eqnarray*}
by Proposition \ref{pr:VG}(i) we obtain that
\begin{eqnarray}\label{VGd}
	\bigg(\sum_{k=1}^p t_k \xi^{(k)}_1\bigg)\xi_2 &\stackrel{d}{=}& \VG\Bigg(1, \sigma_2\sum_{k=1}^p t_k \rho^{(k)} \sigma_1^{(k)},\sigma_2\sqrt{\sum_{k,l=1}^p t_k t_l \sigma_1^{(k)} \sigma_1^{(l)}
		(\rho^{(kl)}-\rho^{(k)} \rho^{(l)} )},0\Bigg). \nonumber\\
\end{eqnarray}
For the right-hand side of \eqref{vd} write
\begin{eqnarray*}
	&& \hskip-1cm\sigma_2\sum_{k=1}^p t_k \sigma_1^{(k)} \big(\rho^{(k)} W_1 + \sqrt{1-(\rho^{(k)})^2}\sqrt{W_1}U_k\big) \\
	&=&\bigg(\sigma_2\sum_{k=1}^p t_k \sigma_1^{(k)} \rho^{(k)}\bigg) W_1 +   \bigg(\sigma_2 \sum_{k=1}^p t_k \sigma_1^{(k)} \sqrt{1-(\rho^{(k)})^2}U_k\bigg)\sqrt{W_1}.
\end{eqnarray*}
Here, by \eqref{UkUl},
\begin{eqnarray*}
	\sigma_2 \sum_{k=1}^p t_k \sigma_1^{(k)} \sqrt{1-(\rho^{(k)})^2} U_k
	&\eqd&
	\sigma_2 \bigg(\sum_{k,l=1}^p t_k t_l \sigma_1^{(k)}\sigma_1^{(l)}
	\sqrt{1-(\rho^{(k)})^2} \sqrt{1-(\rho^{(l)})^2} \E (U_k U_l) \bigg)^{1/2} U_1\\
	&=&
	\sigma_2 \bigg(\sum_{k,l=1}^p t_k t_l \sigma_1^{(k)}\sigma_1^{(l)}
	(\rho^{(kl)}-\rho^{(k)}\rho^{(l)})\bigg)^{1/2} U_1.
\end{eqnarray*}
Note that $U_1\stackrel{d}{=} \N(0,1)$. So that,
\begin{eqnarray*}
	&&\hskip-.8cm\sigma_2\sum_{k=1}^p t_k \sigma_1^{(k)} \big(\rho^{(k)} W_1 + \sqrt{1-(\rho^{(k)})^2}\sqrt{W_1}U_k\big) \\
	&\eqd&
	\bigg(\sigma_2\sum_{k=1}^p t_k \sigma_1^{(k)} \rho^{(k)}\bigg) W_1 +
	\sigma_2 \bigg(\sum_{k,l=1}^p t_k t_l \sigma_1^{(k)}\sigma_1^{(l)}
	(\rho^{(kl)}-\rho^{(k)}\rho^{(l)})\bigg)^{1/2}\sqrt{W_1} U_1,
\end{eqnarray*}
which, by representation \eqref{repr}, has the same VG distribution as that in
\eqref{VGd}. This proves \eqref{vd}. \hfill $\Box$

\section{Some auxiliary lemmas}
\label{sec:auxiliary}

\noindent In this section we establish some auxiliary results that will be used in the proofs of
Theorems~\ref{th:wider} and \ref{th:main}. Here and throughout the paper we remove the upper indices when working with triangular schemes of random variables, e.g., $(V_{1}, \ldots, V_{p}) \equiv (V_{1}^{(p)}, \ldots, V_{p}^{(p)})$, whenever it is clear from the context.

\begin{Lemma}\label{lemma:U1+}
	Let $V = (V_1,\ldots, V_p)' \stackrel{d}{=} \mathcal{N}_p(0,\Sigma_{V}^{(p)})$, where $\Sigma_{V}^{(p)}$ is positive definite covariance matrix and $\operatorname{tr}((\Sigma_{V}^{(p)})^2) = o(p^2)$, $p\to\infty$. Then
	\begin{eqnarray}\label{sd_1}
		\frac{1}{p}\sum_{k=1}^{p}\big(V_k^2 - \mathbb{E}V_{k}^2 \big) &\convP& 0 \ \ as\ \ p \to \infty.
	\end{eqnarray}
	If, in addition, $p^{-1}\operatorname{tr}(\Sigma_{V}^{(p)}) \to 1$, then
	\begin{eqnarray}\label{sd_2}
		\frac{1}{p}\sum_{k=1}^{p}V_k^2 &\convP& 1 \ \ as\ \ p \to \infty.
	\end{eqnarray}
\end{Lemma}

\begin{proof}
	Due to the Spectral Theorem, we have
	\begin{eqnarray}\label{spectral}
		V'V & =& \sum_{k=1}^{p} V_k^2 \ \stackrel{d}{=}\ \sum_{j=1}^p \lambda_j^{(p)} \tilde Z_j^2,
	\end{eqnarray}
	where $\tilde Z_j$ are i.i.d.\ standard normal variables and $\lambda_1^{(p)}, \ldots, \lambda_p^{(p)}$ are the eigenvalues of $\Sigma_{V}^{(p)}$.
	Observe from \eqref{spectral} that
	\begin{eqnarray} \label{spectral2}
		\mathbb{E}V'V & =&  \sum_{j=1}^{p} \lambda_j^{(p)} \ =\ \operatorname{tr}(\Sigma_{V}^{(p)}), \label{eq:evv} \\
		\operatorname{Var}(V'V) &=& \operatorname{Var}
		\Big(\sum_{j=1}^{p} \lambda_j^{(p)} \tilde Z_j^2\Big) \ =\ 2\sum_{j=1}^p(\lambda_j^{(p)})^2 \ =\ 2 \operatorname{tr}((\Sigma_{V}^{(p)})^2). \label{spectral10}
	\end{eqnarray}
	Thus, by (\ref{spectral2})--(\ref{spectral10}),
	for any $\epsilon > 0$
	\begin{eqnarray*}
		\mathbb{P}\Big(\Big|\frac{1}{p}\big( V'V-\mathbb{E}V'V \big)\Big|>\epsilon \Big) &\leq& \frac{\operatorname{Var}(V'V)}{p^2 \epsilon^2} \ \to\ 0, \ \ p\to\infty,
	\end{eqnarray*}
	and the relation in \eqref{sd_1} follows due to assumption ${\rm tr}((\Sigma_{V}^{(p)})^2)~=~o(p^2)$.
	Finally, if $p^{-1}\operatorname{tr}(\Sigma_{V}^{(p)})~\to~1$, by \eqref{eq:evv}, the result \eqref{sd_1} leads to \eqref{sd_2}.
\end{proof}

\begin{remark} \label{remark:trace}
	The assumption on matrix $\Sigma_V=\Sigma^{(p)}_V$ in Lemma~\ref{lemma:U1+}, requiring that $\operatorname{tr}(\Sigma_V^2) = o(p^2)$, is not overly restrictive: assume, for example, that $\Sigma_V = (\sigma^{(i,j)})$ is any KMS type covariance matrix, as in (\ref{KMS}). Then it is straightforward to see that
	\begin{eqnarray*}
		\operatorname{tr}(\Sigma_V^2)&=&\sum_{i,j=1}^p (\sigma^{(i,j)})^2  \ =\
		\sum_{i,j=1}^p \rho^{2|i-j|} \\
		&=& \sum_{|m|<p} (p-|m|) \rho^{2|m|} \ \le\ p\sum_{|m|<p}|m|\rho^{2|m|} \ =\ {\cal O}(p).
	\end{eqnarray*}
\end{remark}

\begin{Lemma} \label{lemma:quadratic}
	Assume that $\tilde Z_1, \tilde Z_2, \ldots$ are i.i.d. $\mathcal{N}(0, 1)$ random variables. For any $p\in {\mathbb N}$ define
	\begin{eqnarray}
		\zeta_j^{(p)} &:=& \nu_j^{(p)} (\tilde Z_j^2 -1) + \gamma_j^{(p)} \sqrt{p} \tilde Z_j, ~~~ j = 1, \ldots, p,
	\end{eqnarray}
	where $\nu_j^{(p)}, ~j = 1, \ldots, p$, are positive scalars, and $\gamma_j^{(p)}$, $j = 1, \ldots, p$, are real scalars, such that
	\begin{eqnarray}\label{eq:assumption1}
		\sum_{j=1}^{p} (\nu_j^{(p)})^3&=& o\bigg(\bigg(\sum_{j=1}^p \Var\big(\zeta_j^{(p)}\big)\bigg)^{3/2}\bigg),\\
		p\sum_{j=1}^{p}(\gamma_j^{(p)})^2 \nu_{j}^{(p)} &=& o\bigg(\bigg(\sum_{j=1}^p \Var\big(\zeta_j^{(p)}\big)\bigg)^{3/2}\bigg)
	\end{eqnarray}
	with $\Var(\zeta_j^{(p)})= 2(\nu_j^{(p)})^2 + p (\gamma_j^{(p)})^2. $
	Then, as $p \to \infty$,
	\begin{eqnarray}
		\frac{\sum_{j=1}^{p}\zeta_j^{(p)}}{\sqrt{\sum_{j=1}^{p}\Var(\zeta_j^{(p)})}} & \convd& \mathcal{N}(0,1). \label{eq:zeta1}
	\end{eqnarray}
\end{Lemma}

\begin{proof}
	The proof uses the method of cumulants and is structured as follows: \begin{enumerate}[label={(\rm \roman*)}]
		\item we establish the moment-generating function of $\zeta_j^{(p)}$,  $M_{\zeta_j^{(p)}}(t) := \mathbb{E}\e^{t\zeta_j^{(p)}}$, and   $\log\big( M_{\zeta_j^{(p)}}(t) \big)$;
		
		\item we find $G(t;p)$
		which corresponds to the cumulant generating function of the sum  $\sum_{j=1}^{p}\zeta_j^{(p)}$;
		
		\item we find $K(t;p) := G\bigg(\displaystyle\frac{t}{\sqrt{\sum_{j=1}^{p}(2( \nu_j^{(p)} )^2 + p (\gamma_j^{(p)})^2 ) }};p\bigg)$, which corresponds to the cumulant generating function of the left hand side of \eqref{eq:zeta1};
		\item finally, in order to prove \eqref{eq:zeta1}, we show that the cumulants $\varkappa_j^{(p)}$, generated by $K(t;p)$, satisfy $\varkappa_1^{(p)} = 0$, $\varkappa_2^{(p)} = 1$ and $\varkappa_d^{(p)} \to 0,~ d=3,4,\dots$, as $p \to \infty$.
	\end{enumerate}
	
	\noindent Step 1. First, rewrite
	\begin{eqnarray}
		\zeta_j^{(p)} &=& \nu_j^{(p)} \bigg(\tilde Z_j + \frac{\gamma_j^{(p)} \sqrt{p} }{2 \nu_j^{(p)} } \bigg)^2 - \nu_j^{(p)} - \frac{(\gamma_j^{(p)})^2 p }{4 \nu_j^{(p)}}. \label{eq:zeta2}
	\end{eqnarray}
	Here, $\psi^{(p)}_j := \Big(\tilde Z_j + \frac{\gamma_j^{(p)} \sqrt{p}}{2 \nu_j^{(p)} } \Big)^2$ has a noncentral chi-squared distribution with the following moment-generating function:
	\begin{eqnarray} \label{eq:zeta3}
		M_{\psi^{(p)}_j}(t) &:=& \mathbb{E}\e^{t\psi^{(p)}_j} \ =\
		(1-2t)^{-1/2} \exp\bigg\{\bigg(\frac{\gamma_{j}^{(p)}}{2 \nu_j^{(p)}} \bigg)^2 tp(1-2t)^{-1}\bigg\}, ~~~ |t| < \frac{1}{2}.
	\end{eqnarray}
	Therefore, by \eqref{eq:zeta2} and \eqref{eq:zeta3},
	\begin{eqnarray*}
		M_{\zeta_j^{(p)}}(t) &=& M_{\psi^{(p)}_j}(\nu_j^{(p)} t) \exp \bigg\{-t \nu_j^{(p)} -tp \bigg(\frac{\gamma_{j}^{(p)}}{2 \nu_j^{(p)}} \bigg)^2\bigg\} \\
		&=& \big(1-2\nu_j^{(p)}t\big)^{-\frac{1}{2}}  \exp\bigg\{\frac{(\gamma_j^{(p)})^2 }{4 \nu_j^{(p)} }\, tp \big(1-2\nu_j^{(p)} t \big)^{-1}-t \bigg( \nu_j^{(p)} + \frac{(\gamma_j^{(p)})^2 p }{4\nu_j^{(p)}}\bigg)\bigg\},
	\end{eqnarray*}
	for $|t| < (2 \nu_j^{(p)})^{-1}$, and
	\begin{eqnarray*}
		\log \big( M_{\zeta_j^{(p)}}(t) \big)  &=& \bigg(\frac{\gamma_{j}^{(p)}}{2 \nu_j^{(p)}} \bigg)^2pt \nu_j^{(p)}  \big( 1 - 2 \nu_j^{(p)} t \big)^{-1} - \frac{1}{2}\log\big(1 - 2 \nu_j^{(p)} t \big)
		- t\bigg( \nu_j^{(p)} + \frac{(\gamma_j^{(p)})^2 p }{4 \nu_j^{(p)} }\bigg) \notag  \\
		&& \hskip -1cm =~~ \frac{1}{2}\big( (\gamma_j^{(p)})^2 p  + 2(\nu_j^{(p)})^2\big) t^2 ~+~ \frac{(\gamma_j^{(p)})^2 p }{2 }  \sum_{k=3}^{\infty}t^k 2^{k-2}  (\nu_j^{(p)})^{k-2} + \frac{1}{2}\sum_{k=3}^{\infty} \frac{2^k (\nu_j^{(p)})^k t^k}{k}. \label{eq:zeta4}
	\end{eqnarray*}
	
	\noindent Step 2. Since $\zeta_1^{(p)},\dots,\zeta_j^{(p)}$ are independent, we have that
	\begin{eqnarray*}
		G(t; p) &=& \sum_{j=1}^{p} \log M_{\zeta_j^{(p)}}(t) ~=~
		\frac{t^2}{2} \sum_{j=1}^{p}\big( (\gamma_j^{(p)})^2 p
		+2(\nu_j^{(p)})^2\big)\\ &&  +\ \frac{p}{2}\sum_{k=3}^{\infty}  2^{k-2} t^k \sum_{j=1}^{p} (\gamma_j^{(p)})^2 (\nu_j^{(p)})^{k-2} ~~+~~  \frac{1}{2}\sum_{k=3}^{\infty}\frac{2^k }{k}t^k\sum_{j=1}^{p}(\nu_j^{(p)})^k.
	\end{eqnarray*}
	\noindent Step 3. It is straightforward to see that
	\begin{eqnarray*}
		K(t;p) &=& G\Bigg(\frac{t}{\sqrt{\sum_{j=1}^{p}\big(2(\nu_j^{(p)})^2 + p(\gamma_j^{(p)})^2 \big)}};p\Bigg) \notag \\
		&=&  \frac{t^2}{2} + \frac{1}{2}\sum_{k=3}^{\infty} 2^{k-2} t^k  \frac{ p\sum_{j=1}^{p} (\gamma_j^{(p)})^2 (\nu_j^{(p)})^{k-2} }{\big(\sum_{j=1}^{p}(2(\nu_j^{(p)})^2 + (\gamma_j^{(p)})^2p )\big)^{k/2}}  \label{eq:K1} \\
		&&  + \ \ \frac{1}{2}\sum_{k=3}^{\infty}\frac{2^k }{k}t^k\frac{ \sum_{j=1}^{p}(\nu_j^{(p)})^k }{\big(\sum_{j=1}^{p}(2(\nu_j^{(p)})^2 + (\gamma_j^{(p)})^2 p )\big)^{k/2}} \label{eq:K2} ~=~\sum_{k=1}^{\infty} \varkappa_{k}^{(p)}\frac{t^k}{k!}, \notag
	\end{eqnarray*}
	where  $\varkappa_1^{(p)} = 0$, $\varkappa_2^{(p)} = 1$, and for $k \geq 3$,
	\begin{eqnarray}
		\varkappa_{k}^{(p)}
		&=& \frac{k! 2^{k-3}  p\sum_{j=1}^{p} (\gamma_j^{(p)})^2 (\nu_j^{(p)})^{k-2} + (k-1)! 2^{k-1}\sum_{j=1}^{p}(\nu_j^{(p)})^k}{\big(\sum_{j=1}^{p}(2(\nu_j^{(p)})^2 + (\gamma_j^{(p)})^2p )\big)^{k/2}}.~~~~~~ \label{eq:K}
	\end{eqnarray}
	
	\noindent Step 4. In order to prove that \eqref{eq:zeta1} holds, it remains to show that, as $p \to \infty$, $\varkappa_{d}^{(p)} \to 0$ for all $d \geq 3$.
	By \eqref{eq:K}, it is equivalent to showing that for any fixed $k \geq 3$, as $p \to \infty$,
	\begin{eqnarray}
		\frac{\sum_{j=1}^{p} (\nu_{j}^{(p)})^{k}}{\big( \sum_{j=1}^{p}\big( 2 (\nu_j^{(p)})^2 + (\gamma_j^{(p)})^2 p\big)\big)^{k/2}} &\to& 0, \label{eq:cumulants_eq1_2_1}\\
		\frac{p\sum_{j=1}^{p}(\gamma_j^{(p)})^2 (\nu_{j}^{(p)})^{k-2}}{\big( \sum_{j=1}^{p}\big( 2 (\nu_j^{(p)})^2 + (\gamma_j^{(p)})^2 p\big)\big)^{k/2}} &\to& 0. \label{eq:cumulants_eq1_2_2}
	\end{eqnarray}
	In order to prove \eqref{eq:cumulants_eq1_2_1} we use induction. The case for $k = 3$ holds by assumption. Assuming that \eqref{eq:cumulants_eq1_2_1} holds for fixed $k\ge3$, we have
	\begin{eqnarray*}
		\frac{\sum_{j=1}^{p} (\nu_{j}^{(p)})^{k+1}} {\big(\sum_{j=1}^{p}\big( 2 (\nu_j^{(p)})^2 + (\gamma_j^{(p)})^2 p\big)\big)^{(k+1)/2}} &\leq& \frac{\big(\sum_{j'=1}^{p}(\nu_{j'}^{(p)})^2\big)^{1/2}\sum_{j=1}^{p} (\nu_{j}^{(p)})^{k}}{\big( \sum_{j=1}^{p}\big( 2 (\nu_j^{(p)})^2 + (\gamma_j^{(p)})^2 p\big)\big)^{(k+1)/2}}  \\
		&\leq&~~ \frac{\big(\sum_{j'=1}^{p}\big(2(\nu_{j'}^{(p)})^2 + (\gamma_{j'}^{(p)})^2 p \big)\big)^{1/2}\sum_{j=1}^{p} (\nu_{j}^{(p)})^{k}}{\big( \sum_{j=1}^{p}\big( 2 (\nu_j^{(p)})^2 + (\gamma_j^{(p)})^2 p\big)\big)^{(k+1)/2}} \\
		&=& \frac{\sum_{j=1}^{p} (\nu_{j}^{(p)})^{k}}{\big( \sum_{j=1}^{p}\big( 2 (\nu_j^{(p)})^2 + (\gamma_j^{(p)})^2 p\big)\big)^{k/2}} ~~~\to~~~ 0,
	\end{eqnarray*}
	concluding that \eqref{eq:cumulants_eq1_2_1} holds for all $k\geq3$. The proof for \eqref{eq:cumulants_eq1_2_2} is analogous: the case for $k = 3$ holds by assumption, thus, we repeat the same arguments as with \eqref{eq:cumulants_eq1_2_1} and conclude that \eqref{eq:cumulants_eq1_2_2} holds for all $k\geq 3$.
	This concludes the proof of the lemma.
\end{proof}

\section{Proof of the main results}	\label{sec:proofs}

\noindent
In this section we give the proofs of theorems \ref{th:wider} and \ref{th:main}.
Throughout the proofs, we express corresponding constants in terms of $\kappa_{i,p}$ and $\kappa_{i}$, $i = 1,2,3$, introduced in \eqref{eq:kappa1}--\eqref{eq:kappa3}. Recall that $\kappa_{i,p} \geq 0$, and, by Remark~\ref{remark1}, $\kappa_{i} < \infty$, for $i=1,2,3$.
\begin{proof}[Proof of Theorem \ref{th:wider}]
	Write
	\begin{eqnarray*}
		\|\mathbb{X}'Y\|_{2}^{2} &=& H_1^2+\dots+H_p^2 \ =:\ H,
	\end{eqnarray*}
	where
	\begin{eqnarray*}
		H_k &:=&\sum_{j=1}^n X_{k,j} \bigg(\sum_{l=1}^p \beta_l X_{l,j}+\varepsilon_j\bigg), \ \ k=1,\dots,p.
	\end{eqnarray*}
	Denote $Z_j:=\sum_{l=1}^p \beta_l X_{l,j}+\varepsilon_j$, $j=1,\dots,n$. By covariance structure \eqref{KMS} and $X_{k,j}\stackrel{d}{=} \N(0,1)$, $\varepsilon_j \eqd
	\mathcal{N}(0, \sigma^2_{\varepsilon})$, we have $Z_j\stackrel{d}{=} \N(0,\sigma^2_Z)$, where $\sigma_Z^2=\sum_{l,l'=1}^p \beta_l\beta_{l'}\rho^{|l-l'|} +\sigma^2_\varepsilon$ and $\Cov(X_{k,j},Z_j)=\sum_{l=1}^p \beta_l \rho^{|k-l|}$.
	
	Applying Proposition~\ref{pr:VG}(iii) with $\xi_{1j}^{(k)}= X_{k,j}$, $\xi_{2j}=Z_j$, and  $\sigma^{(k)}_1=1$, $\sigma_{2,p}=\sigma_Z$, $\theta_{k}^{(p)}:= \rho^{(k)}=\sigma_Z^{-1}\sum_{l=1}^p \beta_l \rho^{|k-l|}$, where $\rho^{(kl)} = \rho^{|k-l|}$, we obtain that
	\begin{eqnarray*}
		\|\mathbb{X}'Y\|_2^2 &\eqd& \sigma_{2,p}^2
		\sum_{k=1}^p \Big(\theta_{k}^{(p)}  W_n +\sqrt{1-(\theta_{k}^{(p)})^2} \sqrt{W_n} U_k\Big)^2,
	\end{eqnarray*}
	where $W_n \stackrel{d}{=} \Gamma(n/2, 1/2)$ and $(U_1, \ldots, U_p)' \stackrel{d}{=} \mathcal{N}(0, \Sigma_U^{(p)})$  with $\Sigma_U^{(p)}=(\sigma_U^{(k,l)})$ defined as (see \eqref{UkUl}):
	\begin{eqnarray}\label{UkUl+}
		\sigma_U^{(k,l)}&=&\frac{\rho^{|k-l|}-\theta_k^{(p)}\theta_l^{(p)}}
		{\sqrt{1-(\theta_k^{(p)})^2}\sqrt{1-(\theta_l^{(p)})^2}}, \ \ k,l=1,\dots,p.
	\end{eqnarray}
	
	\noindent By expanding the square we can write
	\begin{eqnarray*}
		\| \mathbb{X}'Y\|_{2}^{2} &\eqd& \sigma_{2,p}^2 \Big((W_n - \mathbb{E}W_n + \mathbb{E}W_n)^2 \sum_{k=1}^p (\theta_{k}^{(p)})^2 + 2  W_n^{3/2} \sum_{k=1}^{p}\theta_{k}^{(p)} \sqrt{1 - (\theta_{k}^{(p)})^2}  U_k \notag\\
		&&+\ (W_n - \mathbb{E}W_n) \sum_{k=1}^{p} \big(1 - (\theta_{k}^{(p)})^2 \big) U_k^2 +\mathbb{E}W_n \sum_{k=1}^{p} \big(1 - (\theta_{k}^{(p)})^2 \big) U_k^2 \Big) \notag.
	\end{eqnarray*}
	By further rearranging the right-hand side, we have
	\begin{eqnarray} \label{eq:final_1}
		\frac{\| \mathbb{X}'Y\|_{2}^{2}}{n^{3/2}} &\eqd&
		I_{1} + I_{2} + I_{3} + I_{4},
	\end{eqnarray}
	where
	\begin{eqnarray}
		I_{1} &:=& \frac{\sigma_{2,p}^2}{n^{3/2}} (W_n - \mathbb{E}W_n)^2 \sum_{k=1}^{p}(\theta_k^{(p)})^2,   \label{eq:th:I1} \\
		I_2 &:=& \frac{\sigma_{2,p}^2}{n^{3/2}}( W_n - \mathbb{E}W_n) \Big(2 \mathbb{E}W_n \sum_{k=1}^{p}(\theta_k^{(p)})^2 + \sum_{k=1}^{p}(1 - (\theta_k^{(p)})^2)U_k^2 \Big), \label{eq:th:I2} \\
		I_3 &:=& \frac{\sigma_{2,p}^2}{n^{3/2}}\, 2 W_n^{3/2} \sum_{k=1}^{p} \theta_k^{(p)} \sqrt{1-(\theta_k^{(p)})^2} U_k +\frac{\sigma_{2,p}^2}{n^{3/2}}\, \mathbb{E}W_n\sum_{k=1}^{p}\big((1 - (\theta_k^{(p)})^2)U_k^2 -1\big), \label{eq:th:I3} \\
		I_4 &:=&  \frac{\sigma_{2,p}^2}{n^{3/2}}\Big(p \mathbb{E}W_n +
		(\mathbb{E}W_n)^2 \sum_{k=1}^{p}(\theta_k^{(p)})^2 \Big).  \label{eq:th:I4}
	\end{eqnarray}
	
	We will show that, as $p, n \to \infty$, $p/n \to c \in (0, \infty)$, the term $I_1 = o_P(1)$, while the terms $I_2$ and $I_3$ are asymptotically normal. More precisely, we will show that $I_2 \convd\mathcal{N}(0, s_{1}^2)$ and $I_{3} \convd \mathcal{N}(0, s_{2}^2)$, where $s_{1}^2$ and $s_{2}^2$ are given by \eqref{eq:s1} and \eqref{eq:s2} below.   Here, since $W_n$ and $(U_1, \ldots, U_p)'$ are mutually independent for each $n$, it follows that $I_2 + I_3 \stackrel{d}{\to} \mathcal{N}(0, s_1^2 + s_2^2)$. Finally, the term $I_4$ defines the mean of the statistic, i.e.
	\begin{eqnarray}
		\frac{\|\mathbb{X}'Y\|^{2}_{2}}{n^{3/2}} - I_4 &\convd& \mathcal{N}(0, s_1^2 + s_2^2). \label{eq:th:main:1} \label{eq:final_2}
	\end{eqnarray}
	Thus, we will conclude by establishing that $I_4 = \sqrt{n}(\kappa_{2,p} +  p n^{-1} (\kappa_{1,p} + \sigma_{\varepsilon}^2))$, while $s_1^2 + s_2^2 = s^2$, as in the statement of the theorem.
	
	First, consider $I_1$ defined in \eqref{eq:th:I1}. We will show that $I_1=o_P(1)$. Denote
	\begin{eqnarray}\label{skapa}
		c_2 &:=& \lim_{p\to\infty}\sum_{k=1}^p (\theta_{k}^{(p)})^2 ~=~ (\kappa_{1} + \sigma_{\varepsilon}^2)^{-1}\kappa_{2}, ~~~~
		\sigma_{2}^{2}  \ :=\ \lim_{p \to \infty} \sigma_{2,p}^2 ~=~ \kappa_{1} + \sigma_{\varepsilon}^2.
	\end{eqnarray}
	It is clear that $c_2 < \infty$ and $\sigma_2^2 < \infty$. Recall that, by
	CLT,
	\begin{eqnarray}\label{WCLT}
		\frac{W_n-\E W_n}{n^{1/2}}&\stackrel{d}{\to}& \mathcal{N}(0, 2).
	\end{eqnarray}
	Therefore,
	\begin{eqnarray} \label{eq:th1_3} \label{eq:final_3}
		I_{1}&=& \mathcal{O}(1) n^{-1/2} \Big(\frac{W_n - \mathbb{E}W_n}{n^{1/2}}\Big)^2  ~~=~~ o(1) \mathcal{O}_{P}(1) ~~=~~ o_P(1).
	\end{eqnarray}
	
	Second, consider $I_2$, defined in \eqref{eq:th:I2}. We will show that \begin{eqnarray} \label{eq:final_4}
		I_2 ~~\stackrel{d}{\to}~~  \mathcal{N}(0,s_1^{2})
	\end{eqnarray} with $s_{1}^2$ given by
	\begin{eqnarray} \label{eq:s1}
		s_1^2 &=& 2\sigma^4_2 (2c_2 + c)^2 ~~=~~ 8\kappa_2^2 + 8 c(\kappa_{1} + \sigma_{\varepsilon}^2)\kappa_2 + 2c^2(\kappa_{1} + \sigma_{\varepsilon}^2)^2.
	\end{eqnarray}
	Rewrite
	\begin{eqnarray}\label{eq:th:I2+}
		I_2 &=& \sigma_{2,p}^2\,\frac{W_n-\mathbb{E}W_n}{n^{1/2}} \Big(\frac{2\mathbb{E}W_n}{n} \sum_{k=1}^{p}(\theta_k^{(p)})^2 +\frac{1}{n} \sum_{k=1}^{p}(1-(\theta_k^{(p)})^2)U_k^2 \Big).~~~~
	\end{eqnarray}
	Applying \eqref{skapa} and \eqref{WCLT} for the outer term of \eqref{eq:th:I2+}, we obtain
	\begin{eqnarray*}
		\sigma_{2,p}^2\,\frac{W_n-\mathbb{E}W_n}{n^{1/2}}&\stackrel{d}{\to}& \mathcal{N}(0, 2\sigma_{2}^4).
	\end{eqnarray*}
	We will show that the inner term of \eqref{eq:th:I2+} approaches $2c_2 + c$. Since
	$\E W_n=n$, by \eqref{skapa} and assumption $p/n\to c$ it suffices to prove the convergence
	\begin{eqnarray}\label{eq:th:I2_3}
		\frac{1}{p}\sum_{k=1}^p (1 - (\theta_k^{(p)})^2) U_k^2  &\convP& 1.
	\end{eqnarray}
	Denote matrix
	\begin{eqnarray}
		A  &:=&\operatorname{diag}\big(1-(\theta_{1}^{(p)})^2, \ \ldots \ , 1-(\theta_{p}^{(p)})^2\big). \label{eq:theta_1}
	\end{eqnarray}
	To prove \eqref{eq:th:I2_3} we apply Lemma~\ref{lemma:U1+} with $V_j~=~\sqrt{1 - (\theta_{j}^{(p)})^2} U_j $, $j~=~1,\ldots,p$,  and $\Sigma_{V}^{(p)} = A^{1/2} \Sigma_{U} A^{1/2}$. Obviously, the conditions of Lemma \ref{lemma:U1+} will hold if   $\operatorname{tr}((A^{1/2} \Sigma_{U} A^{1/2})^2) = \mathcal{O}(p)$ and $p^{-1}\operatorname{tr}(A^{1/2} \Sigma_{U} A^{1/2}) \to 1$, as $p~\to~\infty$.
	Observe, that
	\begin{eqnarray}
		\operatorname{tr}((A^{1/2} \Sigma_{U} A^{1/2})^2)
		&=& \operatorname{tr}((A\Sigma_{U})^2) \notag\\
		&=&\sum_{k=1}^{p}\sum_{k'=1}^{p}(1 - (\theta_{k}^{(p)})^2)(1 - (\theta_{k'}^{(p)})^2) (\sigma_{U}^{(k,k')})^2 \notag\\
		&=& \sum_{k=1}^{p}\sum_{k'=1}^{p} \big( \rho^{2|k-k'|} - 2 \rho^{|k-k'|}\theta_{k}^{(p)}\theta_{k'}^{(p)} + (\theta_{k}^{(p)})^2(\theta_{k'}^{(p)})^2 \big) \notag \\
		&=& \sum_{k=1}^{p}\sum_{k'=1}^{p}\rho^{2|k-k'|} - 2\left(\kappa_{1,p} + \sigma_{\varepsilon}^{2} \right)^{-1}\kappa_{3,p} + \left(\kappa_{1,p} + \sigma_{\varepsilon}^{2} \right)^{-2}\kappa_{2,p}^2 \notag\\
		&=& \sum_{k=1}^{p}\sum_{k'=1}^{p}\rho^{2|k-k'|} + o(p) ~~\sim~~ p \,\frac{1+\rho^2}{1-\rho^2}, \label{eq:trace_1}
	\end{eqnarray}
	since $\kappa_{i} < \infty$, $i=1,2,3$ and $\kappa_{1,p} \geq 0$. Here we used \eqref{skapa} and the observation that
	\begin{eqnarray} \label{eq:kappa_3_from_trace}
		\sum_{k=1}^{p}\sum_{k'=1}^{p}\rho^{|k-k'|}\theta_{k}^{(p)}\theta_{k'}^{(p)} &=& \frac{\kappa_{3,p}}{\kappa_{1,p}+\sigma_{\varepsilon}^2} ~~\to~~ \frac{\kappa_{3}}{\kappa_{1}+\sigma_{\varepsilon}^2}, ~~\text{ as } p \to \infty.
	\end{eqnarray}

	Similarly, we have
	\begin{eqnarray*}
		\frac{1}{p}\operatorname{tr}(A^{1/2} \Sigma_{U} A^{1/2}) &=& 
		\frac{1}{p}\sum_{k=1}^{p} (1 - (\theta_{k}^{(p)})^2 )~~=~~ 1  ~-~ \frac{\kappa_{2,p}}{p(\kappa_{1,p} + \sigma_{\varepsilon}^2)} ~~\to~~ 1,
	\end{eqnarray*}
	since, by Lemma \ref{lemma:kappa_2p_op}, $\kappa_{2,p} = o(p)$, while $\kappa_{1,p} \geq 0$,  $\kappa_{1}< \infty$. This concludes the proof of \eqref{eq:th:I2_3}.

	Next, consider $I_3$, defined by \eqref{eq:th:I3}. We will show that
	\begin{eqnarray} \label{eq:final_5}
		I_3 &\stackrel{d}{\to}& \mathcal{N}(0,s_2^2),
	\end{eqnarray}
	with $s_2^2$ defined in \eqref{eq:s2}. Write
	\begin{eqnarray*}
		I_3
		&=&  \sigma_{2,p}^{2} \bigg( 2\frac{W_n^{3/2}}{n^{3/2}} \mathbf{b}' U +  n^{-1/2} (U'AU - p) \bigg),
	\end{eqnarray*}
	where ~$U ~=~ \left(U_1, \ldots, U_p \right)'$, $A$ is defined by \eqref{eq:theta_1}, and
	\begin{eqnarray}
		\mathbf{b} &=& \Big(\theta_{1}^{(p)} \sqrt{1 - (\theta_{1}^{(p)})^2}, \ldots, \theta_{p}^{(p)} \sqrt{1 - (\theta_{p}^{(p)})^2}\Big)'. \notag
	\end{eqnarray}
	
	Observe that $n^{-3/2}W_n^{3/2} ~\stackrel{P}{\to}~ 1$ due to the Law of Large Numbers. Thus, since $W_n$ and $U$ are independent for any $n$ and $p/n\to c$, it follows that
	\begin{eqnarray}
		I_3&=& \sigma_{2,p}^{2} \Big(2\mathbf{b}' U + \sqrt{\frac{c}{p}} \left(U'AU - p  \right) \Big) + o_{P}(1).  \label{eq:I3_2}
	\end{eqnarray}
	First, we consider the inner term of \eqref{eq:I3_2} and show, that, as $p ~\to~ \infty$,
	\begin{eqnarray}
		2\mathbf{b}' U  ~~+~~  \sqrt{\frac{c}{p}}\, (UAU' - p)     &\convd&  V_2, \label{eq:naujas_3_ref_1}
	\end{eqnarray}
	where  $V_2 \stackrel{d}{=} \mathcal{N}(0,\sigma_2^{-4} s_2^2)$.  Then, \eqref{eq:final_5} readily follows from \eqref{eq:I3_2}.
	
	Recall, that $U ~\stackrel{d}{=}~ \mathcal{N}_{p}(0,\Sigma_U)$, $\Sigma_U > 0$. Further, let $\tilde Z ~\stackrel{d}{=}~ \mathcal{N}_{p}(0, I_p)$. Clearly, one has that $U ~\stackrel{d}{=}~ \Sigma_U^{1/2} \tilde Z$, where $\Sigma_U^{1/2}$ denotes the symmetric square root of $\Sigma_U$.  By the Spectral Theorem, we construct $V := P' \tilde Z$, where $V \stackrel{d}{=} \mathcal{N}_{p}(0,I_p)$ and  $P$ is an orthogonal matrix that diagonalizes $\Sigma_U^{1/2} A \Sigma_U^{1/2}$, such, that $P' \Sigma_U^{1/2}A\Sigma_U^{1/2}P = \Lambda$, with $\Lambda = \operatorname{diag}(\lambda_1^{(p)}, \ldots, \lambda_p^{(p)})$ comprised of the eigenvalues of $\Sigma_U^{1/2}A\Sigma_U^{1/2}$. Then,
	\begin{eqnarray}
		\frac{\sqrt{c}}{\sqrt{p}} \left(U'AU - p  \right) + 2\mathbf{b}' U      &\stackrel{d}{=}& \frac{\sqrt{c}}{\sqrt{p}} \left( V'\Lambda V- p \right) +  2\mathbf{b}' \Sigma_{U}^{1/2}PV \notag \\
		&=& \frac{\sqrt{c}}{\sqrt{p}}\bigg( \sum_{j=1}^{p}\Big(\lambda_j^{(p)} (V_j^2 - 1) + g_j^{(p)} \sqrt{p} V_j \Big) \bigg) \notag \\
		&=:&  \frac{\sqrt{c}}{\sqrt{p}} \sum_{j=1}^{p}\widetilde V_j^{(p)}, \label{eq:Mj}
	\end{eqnarray}
	where  $(g_1^{(p)}, \ldots, g_p^{(p)})= 2c^{-1/2}\mathbf{b}'\Sigma_{U}^{1/2}P$, and
	\begin{eqnarray}
		\widetilde V_j^{(p)}&:=& \lambda_j^{(p)} (V_j^2 - 1) ~+~ g_j^{(p)} \sqrt{p} V_j, ~~~ j = 1, \ldots, p.
	\end{eqnarray}
	Clearly, $\mathbb{E}\widetilde V_j^{(p)} = 0$ and $\mathbb{E}(\widetilde V_j^{(p)})^2 = 2 (\lambda_j^{(p)})^2 + (g_j^{(p)})^2 p $. Therefore, proving the result \eqref{eq:naujas_3_ref_1} is equivalent to showing:
	\begin{eqnarray}
		\frac{\sqrt{c}}{\sqrt{p}} \sum_{j=1}^{p}\widetilde V_j^{(p)} &\stackrel{d}{\to}& \mathcal{N}(0,\sigma_2^{-4} s_2^2), \label{eq:Mj_4}
	\end{eqnarray}
	where
	\begin{align}
		\sigma_{2}^{-4} s_{2}^2 ~=~ c \lim_{p\to\infty} p^{-1}  \sum_{j=1}^{p}\mathbb{E}(\widetilde V_j^{(p)})^2 ~=~ 2 c \lim_{p\to\infty} p^{-1} \sum_{j=1}^{p} (\lambda_j^{(p)})^2 + c \lim_{p \to \infty} \sum_{j=1}^{p}(g_j^{(p)})^2. \label{eq:s_2}
	\end{align}
	
	We prove \eqref{eq:Mj_4} by applying Lemma \ref{lemma:quadratic} with $\nu_j^{(p)} = \lambda_j^{(p)}$ as the eigenvalues of $\Sigma_U^{1/2}A\Sigma_U^{1/2}$ and $\gamma_j^{(p)} = g_j^{(p)}$. By the conditions of Lemma \ref{lemma:quadratic}, we need to show that the following holds
	\begin{eqnarray}
		\sum_{j=1}^{p}(\lambda_j^{(p)})^3 + p\sum_{j=1}^{p}(g_j^{(p)})^2 \lambda_{j}^{(p)}  &=& o\bigg( \bigg( \sum_{j=1}^{p}\big(2(\lambda_j^{(p)})^2 + (g_j^{(p)})^2 p\big)\bigg)^{3/2} \bigg).~~~~~ \label{eq:conditions_1}
	\end{eqnarray}
	First, observe that $p^{-1}\sum_{j=1}^{p} (2(\lambda_j^{(p)})^2 + (g_j^{(p)})^2 p) \to C \in (0, \infty)$.
	Indeed, we have that $\sum_{j=1}^{p}(g_j^{(p)})^2 \to C_{g} \in (0, \infty)$, since
	\begin{eqnarray}			
		\sum_{j=1}^{p}(g_j^{(p)})^2 &=& 4c^{-1}({\bf b}'\Sigma_U^{1/2} P)({\bf b}'\Sigma_U^{1/2} P)' \ =\  4c^{-1}{\bf b}'\Sigma_U {\bf b}\notag\\
		&=&  4c^{-1}\sum_{j=1}^{p}\sum_{j'=1}^{p}\theta_j^{(p)} \theta_{j'}^{(p)} \sqrt{1 - (\theta_j^{(p)})^2}    \sqrt{1 - (\theta_{j'}^{(p)})^2} \sigma_U^{(j,j')}\notag \\
		&=& 4c^{-1}\sum_{j=1}^{p}\sum_{j'=1}^{p}\theta_{j}^{(p)}\theta_{j'}^{(p)} \left(\rho^{|j-j'|}-\theta_{j}^{(p)} \theta_{j'}^{(p)} \right)    \notag \\
		&\to& 4c^{-1}\left(\kappa_{1} + \sigma_{\varepsilon}^2 \right)^{-1} \kappa_3 - 4c^{-1}\left(\kappa_{1} + \sigma_{\varepsilon}^2 \right)^{-2}\kappa_{2}^2   ~~=~~ C_g \label{eq:main_th_gj}
	\end{eqnarray}
	by \eqref{skapa} and \eqref{eq:kappa_3_from_trace}.
	
	Next, by  \eqref{eq:trace_1}, we find that $p^{-1}\sum_{j=1}^{p} (\lambda_j^{(p)})^2 \to C_{\lambda} \in (0,\infty)$. Indeed, by \eqref{eq:trace_1}, we have
	\begin{eqnarray}
		\sum_{j=1}^{p} (\lambda_j^{(p)})^2 &=&\operatorname{tr}{((\Sigma_{U}^{1/2}A \Sigma_U^{1/2})^2)} ~=~ \operatorname{tr}{((\Sigma_{U} A)^2) } \notag \\
		&=& \sum_{j=1}^{p}\sum_{j'=1}^{p}\rho^{2|j-j'|} + o(p) ~~\sim~~ p \,\frac{1+\rho^2}{1-\rho^2}. \label{eq:lemma_42_used}
	\end{eqnarray}
	Thus, by \eqref{eq:main_th_gj} and \eqref{eq:lemma_42_used}, it follows that $p^{-1}\sum_{j=1}^{p} (2c(\lambda_j^{(p)})^2 + (g_j^{(p)})^2 p) \to C \in (0, \infty)$ and condition \eqref{eq:conditions_1} reduces to:
	\begin{eqnarray}
		\sum_{j=1}^{p}(\lambda_j^{(p)})^3 + p\sum_{j=1}^{p}(g_j^{(p)})^2 \lambda_{j}^{(p)} ~~=~~ o(p^{3/2}).	\label{eq:conditions_3}
	\end{eqnarray}

	\noindent We show that \eqref{eq:conditions_3} holds. For the first term of \eqref{eq:conditions_3}, we have
	\begin{eqnarray}
		\sum_{j=1}^{p}(\lambda_j^{(p)})^3 &=& \operatorname{tr}((\Sigma^{1/2}_U A \Sigma^{1/2}_U)^3) = \operatorname{tr}((\Sigma_U A)^3) \notag  \\
		&=& \sum_{i,j,k=1}^{p} \big(1 - (\theta_{i}^{(p)})^2 \big)\big(1 - (\theta_{k}^{(p)})^2 \big)\big(1 - (\theta_{j}^{(p)})^2 \big) \sigma_U^{(i,j)}\sigma_U^{(i,k)}\sigma_U^{(k,j)} \notag  \\
		&=& \sum_{i,j,k=1}^{p}\big(\rho^{|i-j|} + \theta_{i}^{(p)} \theta_{j}^{(p)}  \big)\big(\rho^{|i-k|} + \theta_{i}^{(p)} \theta_{k}^{(p)}\big)\big(\rho^{|k-j|} + \theta_{k}^{(p)} \theta_{j}^{(p)}  \big) \notag \\
		&=& o(p^{3/2}), \label{eq:lambda_3_op}
	\end{eqnarray}
	where the last equality follows from Lemma \ref{lemma:lambda_3}.
	For the second term of \eqref{eq:conditions_3}, observe, that by  H\"older's inequality and \eqref{eq:lambda_3_op},
	\begin{eqnarray*}
		p\sum_{j=1}^{p}(g_j^{(p)})^2 \lambda_{j}^{(p)} &\leq& p \bigg(\sum_{j=1}^{p}|g_j^{(p)}|^3\bigg)^{2/3}\bigg(\sum_{j=1}^{p}(\lambda_{j}^{(p)})^{3}\bigg)^{1/3} \\
		&=& p^{3/2} \mathcal{O}(1)\bigg(\frac{\sum_{j=1}^{p}(\lambda_{j}^{(p)})^{3}}{p^{3/2}}\bigg)^{1/3} ~~=~~ o(p^{3/2}).
	\end{eqnarray*}
	This concludes with \eqref{eq:conditions_3}, ensuring that the conditions of Lemma \ref{lemma:quadratic} hold.
	
	Now we can establish the expression for $s_{2}^2$. By \eqref{skapa}, \eqref{eq:s_2}, \eqref{eq:main_th_gj} and \eqref{eq:lemma_42_used},
	\begin{eqnarray}
		s_{2}^2 &=&  \sigma_{2}^{4}\lim_{p\to\infty}\sum_{j=1}^{p}\big(2p^{-1}c(\lambda_j^{(p)})^2 + c(g_j^{(p)})^2 \big) \notag \\
		&=&  \sigma_{2}^{4} \lim_{p\to\infty} \frac{2c}{p}\bigg( \sum_{k=1}^{p}\sum_{k'=1}^{p}\rho^{2|k-k'|} + o(p) \bigg) + 4\sigma_{2}^{4} \left(\kappa_{1} + \sigma_{\varepsilon}^2 \right)^{-1} \kappa_3 - 4\sigma_{2}^{4}\left(\kappa_{1} + \sigma_{\varepsilon}^2 \right)^{-2}\kappa_{2}^2  \notag \\
		&=&2c\,\frac{1+\rho^2}{1-\rho^2}\, (\kappa_{1} + \sigma_{\varepsilon}^2)^2 +   4 (\kappa_{1} + \sigma_{\varepsilon}^2 )\kappa_3 - 4\kappa_2^{2}.  \label{eq:s2} \label{eq:final_7} \label{eq:Mj_3}
	\end{eqnarray}
	\noindent By \eqref{eq:s1} and \eqref{eq:s2}, recalling that $s^2= s_1^2 + s_2^2$, we have that
	\begin{eqnarray}
		s^2&=&
		4 \kappa_{2}^2 + 4(\kappa_{1} + \sigma_{\varepsilon}^2)\left(2 \kappa_{2} c + \kappa_{3} \right) + 2c (\kappa_{1} + \sigma_{\varepsilon}^2)^2  \Big(c+\frac{1+\rho^2}{1-\rho^2}\Big).~~~~~~~~
		\label{eq:sigma_v}
	\end{eqnarray}
	
	Finally, consider $I_4$, defined by \eqref{eq:th:I4}. Since $\mathbb{E}W_n = n$, we have that
	\begin{eqnarray}
		I_4 &=&
		\frac{\kappa_{1,p}+\sigma^2_\varepsilon}{n^{3/2}}
		\Big(n^2\,\frac{\kappa_{2,p}}{\kappa_{1,p}+\sigma^2_\varepsilon}+pn\Big) ~=~ \sqrt{n}\Big(\kappa_{2,p} +  \frac{p}{n} (\kappa_{1,p} + \sigma_{\varepsilon}^2)\Big). \label{eq:I4}
	\end{eqnarray}
	
	By \eqref{eq:final_1}, having established 4 parts by \eqref{eq:th:I1}--\eqref{eq:th:I4}, we proved that \eqref{eq:final_2} holds due to
	\eqref{eq:final_3}, 
	\eqref{eq:final_4}, 
	\eqref{eq:final_5}, 
	\eqref{eq:final_7}, 
	with terms \eqref{eq:sigma_v} and \eqref{eq:I4}, as in the statement of the theorem, thus concluding the proof.
\end{proof}

Before proceeding with the proof of Theorem \ref{th:main}, we establish the following lemma that ensures $\mathcal{O}(p^{-1/2})$ convergence rate for $\kappa_{1,p}$ and $\kappa_{2,p}$, appearing in Theorem~\ref{th:wider}, under additional restrictions for the parameters $\beta_j$.
\begin{Lemma} \label{lemma:op}
	Assume that $\sum_{j=p+1}^{\infty}\beta_j^2 = o(p^{-1/2})$ and $\sup_{j \geq 1} |\beta_j| j^{\alpha} < \infty$, $\alpha > 1/2$, and $|\rho|<1$. Then,
	\begin{enumerate}[label={(\rm \roman*)}]
		\item $\begin{aligned}
			\kappa_{1} ~=~ \kappa_{1,p} + o(p^{-1/2}),
		\end{aligned}$
		\item $\begin{aligned}
			\kappa_{2} ~=~ \kappa_{2,p} + o(p^{-1/2}).
		\end{aligned}$
	\end{enumerate}
\end{Lemma}
\begin{proof}
	For the proof see Appendix \ref{sec:B1}.
\end{proof}

\begin{proof}[Proof of Theorem \ref{th:main}]
	Rewrite the left-hand side of \eqref{eq:th_2} as follows:
	\begin{eqnarray*}
		\frac{\|\mathbb{X}'Y\|_2^{2} - n^2(\kappa_{2} +  c (\kappa_{1} + \sigma_{\varepsilon}^2))}{n^{3/2}} &=&
		\frac{\|\mathbb{X}'Y\|_2^{2} - n^2\kappa_{2,p} - pn(\kappa_{1,p} + \sigma_{\varepsilon}^2)}{n^{3/2}} \\
		&& +~ \sqrt{n}(\kappa_{2,p} -\kappa_{2}) + \sqrt{n}c(\kappa_{1,p}-\kappa_{1}) + o(1).
	\end{eqnarray*}
	It remains to apply Lemma \ref{lemma:op} and Theorem \ref{th:wider} in order to conclude the proof of the theorem.
\end{proof}

\noindent
We end this section by deriving two supporting results that allows us to derive convenient alternative expressions for the terms $\kappa_1, \kappa_2$ and $\kappa_3$. For this, we introduce functions $\beta(\cdot)$ and $b(\cdot)$ by Definition \ref{def:main} below, which, under the assumptions of Theorem \ref{th:wider} and a  given structure of $\beta_j$'s, requires only to evaluate the terms $\beta(1), \beta(\rho), \beta(\rho^2)$ and $b_{1}(\rho), b_{2}(\rho)$. Then, due to Lemma \ref{lemma:alt_expressions} below, the expressions for $\kappa_1$, $\kappa_2$ and $\kappa_3$ easily follow.

\begin{definition}\label{def:main} Assume that $\sum_{j=1}^{\infty}\beta_{j}^2 < \infty$ and $|\rho| \leq 1$. Define,
	\begin{eqnarray}
		\beta(\rho)&:=&\sum_{j=1}^\infty \beta_j^2 \rho^j, \label{beta} \\
		b_{1}(\rho)&:=& \sum_{j'=2}^{\infty}\sum_{j=1}^{j'-1} \beta_j \beta_{j'}\rho^{j'-j}, \\
		b_{2}(\rho)&:=& \sum_{j=2}^{\infty}\sum_{j'=1}^{j-1} \beta_j \beta_{j'}\rho^{j+j'},\label{gamma}
	\end{eqnarray}
	and define the following quantities which involve derivatives of \eqref{beta}--\eqref{gamma}:	\begin{eqnarray} \label{eq:beta_d}
		\beta^{(1)}(\rho) &:=& \rho \frac{ {\rm d}\beta(\rho)  }{{\rm d} \rho} ~~=~~ \sum_{j=1}^\infty j \beta_j^2 \rho^j , \label{eq:beta1}\\
		b_{1}^{(1)}(\rho) &:=& \rho \frac{{\rm d} b_{1}(\rho)}{{\rm d}\rho} ~~=~~ \sum_{j'=2}^{\infty}\sum_{j=1}^{j'-1} \beta_j \beta_{j'}\rho^{j'-j}(j'-j), \label{gamma_1} \\
		b_{2}^{(1)}(\rho) &:=& \rho \frac{{\rm d} b_{2}(\rho)}{{\rm d}\rho} ~~=~~ \sum_{j'=2}^{\infty}\sum_{j=1}^{j'-1} \beta_j \beta_{j'}\rho^{j'+j}(j'+j), \\
		b^{(2)}(\rho) &:=& \rho^2 \frac{{\rm d}^2 b_{1}(\rho) }{{\rm d} \rho^2} +  b_{1}^{(1)}(\rho) ~~=~~ \sum_{j'=2}^{\infty}\sum_{j=1}^{j'-1} \beta_j \beta_{j'}\rho^{j'-j}(j'-j)^2.~~~~~~~~~ \label{gamma_2}
	\end{eqnarray}
\end{definition}

Note, that, by the rules of differentiation of power series, the functions \eqref{eq:beta1}--\eqref{gamma_2} are well defined.

\noindent
\begin{Lemma} \label{lemma:alt_expressions}
	Let the assumptions of Theorem \ref{th:wider} hold. Let $\kappa_1$, $\kappa_2$ and $\kappa_3$ be given by (\ref{eq:kappa1}), (\ref{eq:kappa2}) and (\ref{eq:kappa3}), respectively. Then, under notation in Definition \ref{def:main}, the following identities hold:
	\begin{enumerate}[label={(\rm \roman*)}]
		\item $\begin{aligned}[t]
			\kappa_1 &~=~  \beta(1) ~+~ 2 b_{1}(\rho),
		\end{aligned}$
		\item $\begin{aligned}[t] \label{eq:const1}
			\kappa_2 &~=~  \beta(1) \frac{1+\rho^2}{1-\rho^2} ~-~\beta(\rho^2)\frac{1}{1-\rho^2} ~+~ 2 \Big( b_{1}^{(1)}(\rho) ~+~ b_{1}(\rho)\frac{1+\rho^2}{1-\rho^2} ~-~ b_{2}(\rho)\frac{1}{1-\rho^2}\Big),
		\end{aligned}$
		\item $\begin{aligned}[t] \kappa_3
			&=~	\frac{1}{(1-\rho^2)^2}\big((1+4\rho^2+\rho^4)(\beta(1) + 2b_{1}(\rho))  -(1+3\rho^2)(\beta(\rho^2) + 2 b_{2}(\rho))\big) \\
			&\ \ \ \ + \frac{1}{1-\rho^2}\big(3b_{1}^{(1)}(\rho)(1+\rho^2) - 2\big(b_{2}^{(1)}(\rho)+ \beta^{(1)}(\rho^2)\big)\big) + b^{(2)}(\rho).
		\end{aligned}$
	\end{enumerate}

\end{Lemma}
\begin{proof}
	See the proof in Appendix \ref{ap:lemma45}.
\end{proof}

\begin{remark} \label{remark1}
	From the assumptions of Definition \ref{def:main} it follows that $\beta(1), |\beta(\rho)|, |b_{1}(\rho)|$, $|b_{2}(\rho)| < \infty$ for $|\rho|<1$. Thus, it follows from Lemma \ref{lemma:alt_expressions} that  $\kappa_i < \infty$, $i = 1,2,3$.
\end{remark}
\begin{proof}[Proof of Remark \ref{remark1}] Cases for $\beta(1)$ and $\beta(\rho)$ follow straightforwardly from the assumptions. Consider $b_1(\rho)$. Note, that for any $p$,
	\begin{eqnarray*}
		|b_{1}(\rho)| &\leq& \sum_{l_1,l_2 = 1}^{\infty} |\beta_{l_1}| |\beta_{l_2}| |\rho|^{|l_1-l_2|} ~~=~~ \sum_{l_1,l_2=1}^{\infty} \big(|\beta_{l_1}| |\rho|^{|l_1-l_2|/2}\big)\big(|\beta_{l_2}| |\rho|^{|l_1-l_2|/2}\big)  \\
		&\le&  (1/2)\sum_{l_1,l_2=1}^{\infty} \big(\beta^2_{l_1}|\rho|^{|l_1-l_2|} ~+~
		\beta^2_{l_2} |\rho|^{|l_1-l_2|}\big)\\
		&=& \sum_{l_1=1}^{\infty} \beta^2_{l_1}  \sum_{l_2=1}^{\infty}|\rho|^{|l_1-l_2|} \ \leq \ \beta(1) \frac{1+|\rho|}{1-|\rho|}\  <\ \infty
	\end{eqnarray*}
	by (\ref{uii}). In a similar manner, it is easy to see that $|b_{2}(\rho)| \leq \beta(1)\frac{|\rho|}{1-|\rho|}$.
\end{proof}

\section{Approximate sparsity: an example}
\label{sec:sparsity}

\noindent
In this section we study the case when coefficients $\beta_j$ decay hyperbolically, i.e., $\beta_j = j^{-1}, j \geq 1$. This assumption is analogous to the assumption of approximate sparsity, as defined by \cite{Chernozhukov2013}. The authors of the aforementioned paper note, that for approximately sparse models the regression function can be well approximated by a linear combination of relatively few important regressors, which is one of the reasons of  popularity of variable selection approaches such as LASSO (\cite{Tibshirani}) and it's modifications (see, e.g., \cite{zou2006adaptive}, \cite{MEINSHAUSEN2007374}, \cite{belloni2011square}). At the same time, approximate sparsity allows all coefficients $\beta_j$ to be nonzero, which is a more plausible assumption in many real world settings.

In order to derive the quantities in Theorem~\ref{th:main}, we apply the results of Lemma~\ref{lemma:alt_expressions}. For this, we establish the expressions for the quantities in Definition~\ref{def:main}.

Define the real dilogarithm function (see, e.g., \cite{10.2307/2006312}): 
\begin{eqnarray} \label{eq:polylog}
	\operatorname{Li}_{2}(x) &=& - \int_{0}^{x}\frac{\log(1-u)}{u} \dd u, ~~~ x\leq1, ~~x \in \mathbb{R}.
\end{eqnarray}
(Here and below, $\int_{0}^{x} = - \int_{x}^{0}$ if  $x \leq 0$.) For $|x| \leq 1$ the real dilogarithm has a series representation,
\begin{eqnarray}
	\operatorname{Li}_{2}(x) &=& \sum_{k=1}^{\infty}\frac{x^{k}}{k^2}.
\end{eqnarray}
Then,
\begin{eqnarray*}
	\beta(1) ~=~ \sum_{j=1}^{\infty} \frac{1}{j^2} ~=~ \frac{\pi^2}{6},&&\beta(\rho) ~=~\sum_{j=1}^{\infty} \frac{\rho^{j}}{j^2} ~=~  \text{Li}_2(\rho).
\end{eqnarray*}
Additionally, we have
\begin{eqnarray} \label{eq:polylog_derivative}
	\frac{\dd}{\dd \rho}\operatorname{Li}_{2}(\rho) &=& -\frac{\log(1-\rho)}{\rho}.
\end{eqnarray}
Thus, by \eqref{eq:beta_d} and \eqref{eq:polylog_derivative}, we establish
\begin{eqnarray*}
	\beta^{(1)}(\rho) &=& \rho \frac{\dd}{\dd \rho}\beta(\rho) ~=~ \rho \frac{\dd}{\dd \rho} \operatorname{Li}_{2}(\rho) ~=~ -\log(1-\rho).
\end{eqnarray*}
Next, note that
\begin{eqnarray}
	b_{1}(\rho)&=&\sum_{i=2}^{\infty} \sum_{j=1}^{i-1}\frac{\rho^{i-j}}{ij} = \sum_{i=2}^{\infty} \sum_{k=1}^{i-1}\frac{\rho^k}{i(i-k)}\notag\\
	&=& \sum_{k=1}^\infty \rho^k \sum_{i=k+1}^\infty\frac{1}{i(i-k)} \ =\
	\sum_{k=1}^\infty \frac{\rho^k}{k} \sum_{l=1}^k\frac{1}{l}\notag\\
	&=&
	\sum_{l=1}^\infty  \frac{1}{l} \sum_{k=l}^\infty \frac{\rho^k}{k} \ =\
	\sum_{l=1}^\infty \frac{1}{l} \int_0^\rho \frac{x^{l-1}}{1-x}\dd x \notag\\
	&=& -\int_{0}^{\rho}  \frac{\log (1-x)}{x(1-x)} \dd x 
	\ = \ \frac{\log^2(1-\rho)}{2} + \text{Li}_{2}(\rho), \label{eq:5_1}
\end{eqnarray}
where we have used identities
\begin{eqnarray*}
	\sum_{i=k+1}^\infty\frac{1}{i(i-k)} & =& \frac{1}{k} \sum_{l=1}^k\frac{1}{l}, \ k\ge1, \ \ {\rm ~~~~~~}\ \ \sum_{k=l}^\infty \frac{\rho^k}{k} =\int_0^\rho \frac{x^{l-1}}{1-x}\dd x
\end{eqnarray*}
and \eqref{eq:polylog}.
Then, by \eqref{gamma_1}, \eqref{eq:polylog_derivative} and \eqref{eq:5_1},
\begin{eqnarray*}
	b_{1}^{(1)}(\rho) &=& \rho \frac{\dd}{\dd \rho}b_{1}(\rho)  ~=~ -\frac{\log (1-\rho)}{1-\rho},
\end{eqnarray*}
whereas by (\ref{gamma_2}),
\begin{eqnarray*}
	b^{(2)}(\rho) &=&\rho^2 \frac{{\rm d}^2 b_{1}(\rho) }{{\rm d} \rho^2} +  b_{1}^{(1)}(\rho)
	~~=~~ \frac{\rho-\rho \log (1-\rho)}{(1-\rho)^2}.
\end{eqnarray*}
\noindent	Further, note that
\begin{eqnarray}
	b_{2}(\rho)&=& \sum_{i=2}^{\infty} \sum_{j=1}^{i-1}\frac{\rho^{i+j}}{ij} ~=~\sum_{i=2}^{\infty} \frac{\rho^i}{i} \sum_{j=1}^{i-1} \frac{\rho^j}{j} ~~=~~\sum_{i=2}^{\infty} \frac{\rho^{i}}{i}  \int_{0}^{\rho} \sum_{j=1}^{i-1} x^{j-1} \dd x \notag \\
	&=&\sum_{i=1}^{\infty} \frac{\rho^{i+1}}{i+1}  \int_{0}^{\rho}  \frac{1-x^{i}}{1-x}   \dd x  \notag\\
	&=&-\log(1-\rho)\bigg(\sum_{i=1}^{\infty} \frac{\rho^{i}}{i} - \rho\bigg)   - \int_{0}^{\rho}\left(\sum_{i=1}^{\infty} \frac{\rho^{i}}{i}    \frac{x^{i-1}}{1-x} - \rho\frac{1}{1-x}  \right)\dd x  \notag\\
	&=&-\log(1-\rho)\sum_{i=1}^{\infty} \frac{\rho^{i}}{i}   - \int_{0}^{\rho}\sum_{i=1}^{\infty} \frac{(\rho x)^{i}}{i}    \frac{1}{x(1-x)} \dd x  \notag\\
	&=& \log^2(1-\rho) + \int_{0}^{\rho} \frac{\log (1-\rho x)}{ x(1-x)}  \dd x~~~~~~\notag \\
	&=&  \frac{1}{2}\big(\log^2(1-\rho) - \operatorname{Li}_{2}(\rho^2)  \big),~~~~~~~~ \label{eq:gamma_2_sparse}
\end{eqnarray}
where the last equality follows from Lemma \ref{lemma:integral2}.
Next, by \eqref{gamma_1}, \eqref{eq:polylog_derivative} and \eqref{eq:gamma_2_sparse} we have
\begin{eqnarray*}
	b_{2}^{(1)}(\rho)
	&=& \log \left(1-\rho^2\right)-\frac{\rho \log (1-\rho)}{1-\rho}.
\end{eqnarray*}

\noindent Thus, we can apply Lemma \ref{lemma:alt_expressions}(i) and arrive at the following expression for $\kappa_1$:
\begin{eqnarray} \label{eq:s_sparse}
	\kappa_1 &=&  \frac{\pi^2}{6} ~+~ \log^2(1-\rho)  ~+~2\text{Li}_{2}(\rho).
\end{eqnarray}
Similarly, for $\kappa_2$, by collecting and simplifying the terms, by Lemma \ref{lemma:alt_expressions}(ii) and Lemma \ref{lemma:integral2},  we have
\begin{eqnarray}\label{eq:kappa2_sparse}
	\kappa_2
	&=& \frac{1+\rho^2}{1-\rho^2}\Big(\frac{\pi^2}{6} + 2\text{Li}_{2}(\rho) \Big) ~-~ \frac{2\log (1-\rho)}{1-\rho}  ~+~ \log^2(1-\rho)\frac{\rho^2}{1-\rho^2} \notag \\
	&=& \frac{1}{1-\rho^2}\big((1+\rho^2)\kappa_{1} - \log^2(1-\rho) - 2(1+\rho)\log(1-\rho){\big)}.
\end{eqnarray}
Lastly, for $\kappa_3$, by Lemma \ref{lemma:alt_expressions}(iii), through simplification of terms, we get
\begin{eqnarray}
	\kappa_{3}
	&=& \frac{1}{(1-\rho^2)^2}\Big((1+4\rho^2+\rho^4)\Big(\frac{\pi^2}{6}  + 2\text{Li}_{2}(\rho)\Big) +  \log^2(1-\rho)\rho^2(1+\rho^2) \notag  \\
	&& -~~ (3-\rho+4\rho^2)(1+\rho)\log(1-\rho)  + \rho(1+\rho)^2 \Big) \notag \\
	&=& \kappa_{2}~\frac{1+3\rho^2}{1-\rho^2} + \frac{1}{(1-\rho^2)^2}\Big(  (-1+\rho+2\rho^2)(1+\rho)\log(1-\rho)  + \rho(1+\rho)^2 -2\rho^4 \kappa_{1} \Big).~~~~~~~~~~ \label{eq:kappa3_sparse}
\end{eqnarray}

This allows us to apply Theorem \ref{th:main} under the considered specification of the parameter $\beta$, and conclude with the following corollary.

\begin{Corollary} \label{cor:main}
	Assume a model \eqref{model} with \eqref{KMS} covariance structure and consider $\beta_j := j^{-1}$, $j = 1, \ldots, p$. Let $p=p_n$ satisfies
	\begin{eqnarray*}
		p\to\infty,  && \frac{p}{n}\to c\in (0,\infty).
	\end{eqnarray*}
	Then
	\begin{eqnarray}\label{eq:th_sparse}
		\frac{\|\mathbb{X}'Y\|_2^{2} - n^2\big(\kappa_{2} +  c (\kappa_{1} + \sigma_{\varepsilon}^2)\big)}{n^{3/2}} &\convd& \mathcal{N}(0,s^2),
	\end{eqnarray}
	where
	\begin{eqnarray}
		s^2&=&	4 \kappa_{2}^2 + 4(\kappa_{1} + \sigma_{\varepsilon}^2)\left(2 \kappa_{2} c + \kappa_{3} \right) + 2c (\kappa_{1} + \sigma_{\varepsilon}^2)^2  \Big(c + \frac{1+\rho^2}{1-\rho^2}\Big),~~~~~~
	\end{eqnarray}
	and $\kappa_1$, $\kappa_2$ and $\kappa_3$ are defined by \eqref{eq:s_sparse}, \eqref{eq:kappa2_sparse} and \eqref{eq:kappa3_sparse}, respectively.
\end{Corollary}

In order to illustrate the results of Corollary~\ref{cor:main}, we end this section with a Monte Carlo simulation study, where we generate 1000 independent replications of the statistic $\|\mathbb{X}'Y\|^{2}_{2}$. The data is generated following the assumptions of Corollary~\ref{cor:main} for varying sets of parameters $(n,p,\rho,\sigma_{\varepsilon}^{2})$. The results are presented in Figures \ref{fig:2}--\ref{fig:6}. Figures show the empirical cumulative distribution function (CDF) and the empirical probability density function (PDF), together with the limiting CDF and PDF of $V \stackrel{d}{=}\mathcal{N}(0, s^2)$ for different parameter combinations. We notice that for relatively small values of $\rho$, the distribution is fairly close to the limiting distribution even for small values of $(p,n)$. On the other hand, slightly slower convergence is evident for $\rho \approx 1$ (see Figure \ref{fig:2} for simulation results with $\rho = 0.95$). However, it can be noted that for large values of $\rho$ the effect of $c$ term is greatly reduced, resulting in very similar distributions when comparing, e.g., $c =1 $ against $c = 10$.

\begin{figure}
	\centering
	\includegraphics[width=0.49\linewidth]{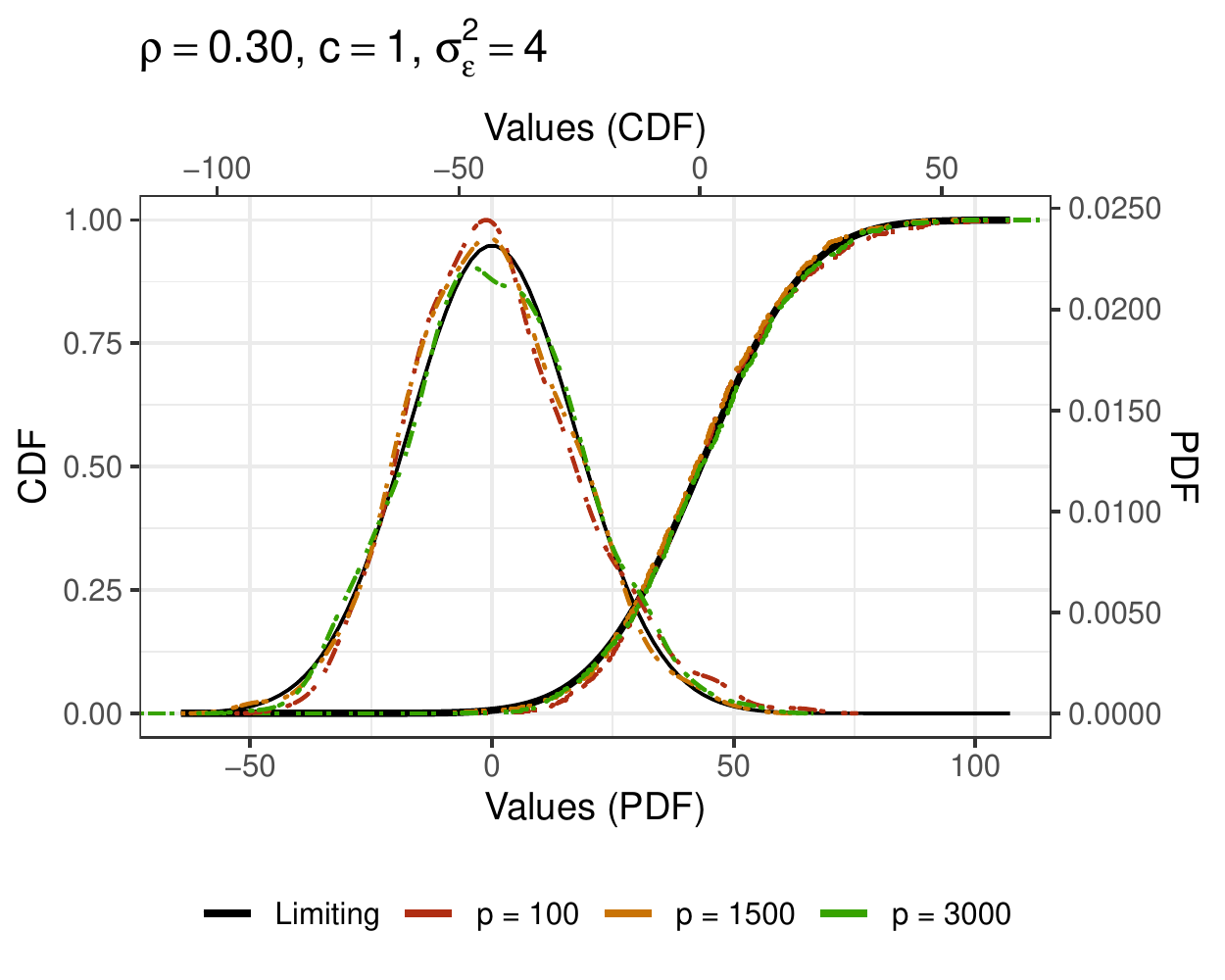} \includegraphics[width=0.49\linewidth]{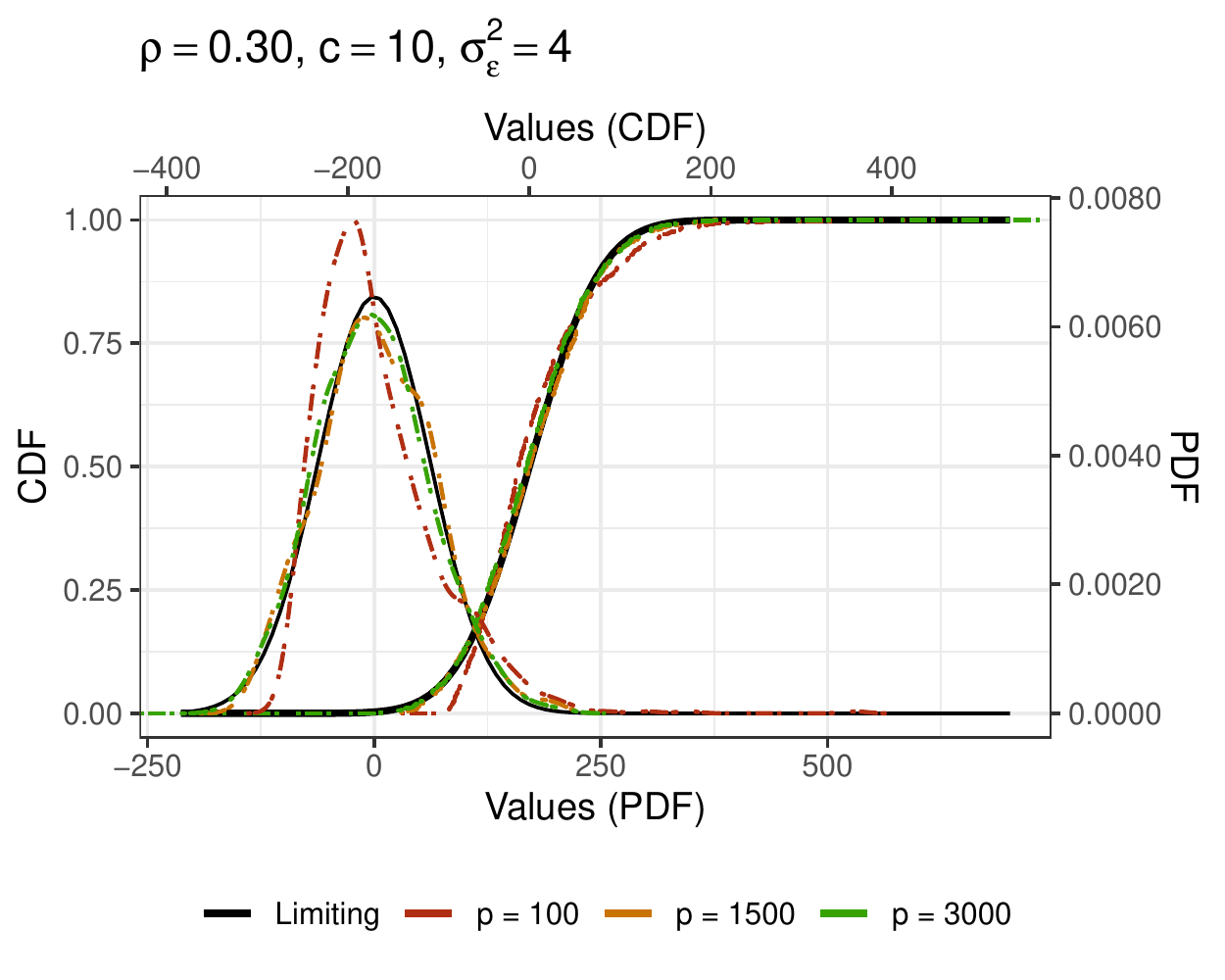}
	\caption{Comparison of the PDF and CDF using 1000 replications from the Monte Carlo simulation of the statistic \eqref{eq:th_sparse} with the limiting distribution $\mathcal{N}(0,s^2)$ by the Corollary \ref{cor:main} (in black) for $\rho = 0.3$, $c=1$, $\sigma_{\varepsilon}^2 = 4$ (left) and $\rho = 0.3$, $c=10$, $\sigma_{\varepsilon}^2 = 4$ (right). }
	\label{fig:2}
\end{figure}
\begin{figure}
	\centering
	\includegraphics[width=0.49\linewidth]{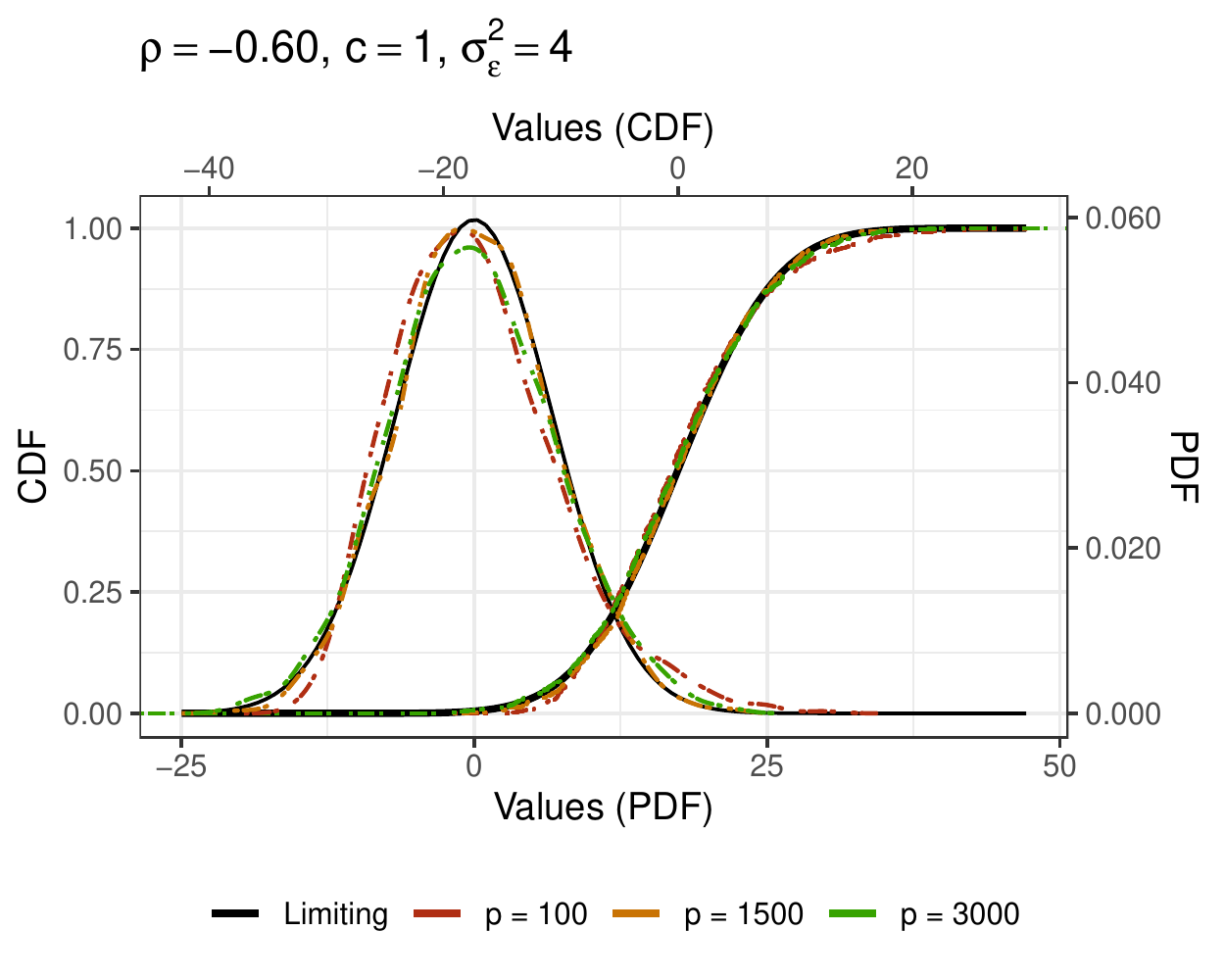} \includegraphics[width=0.49\linewidth]{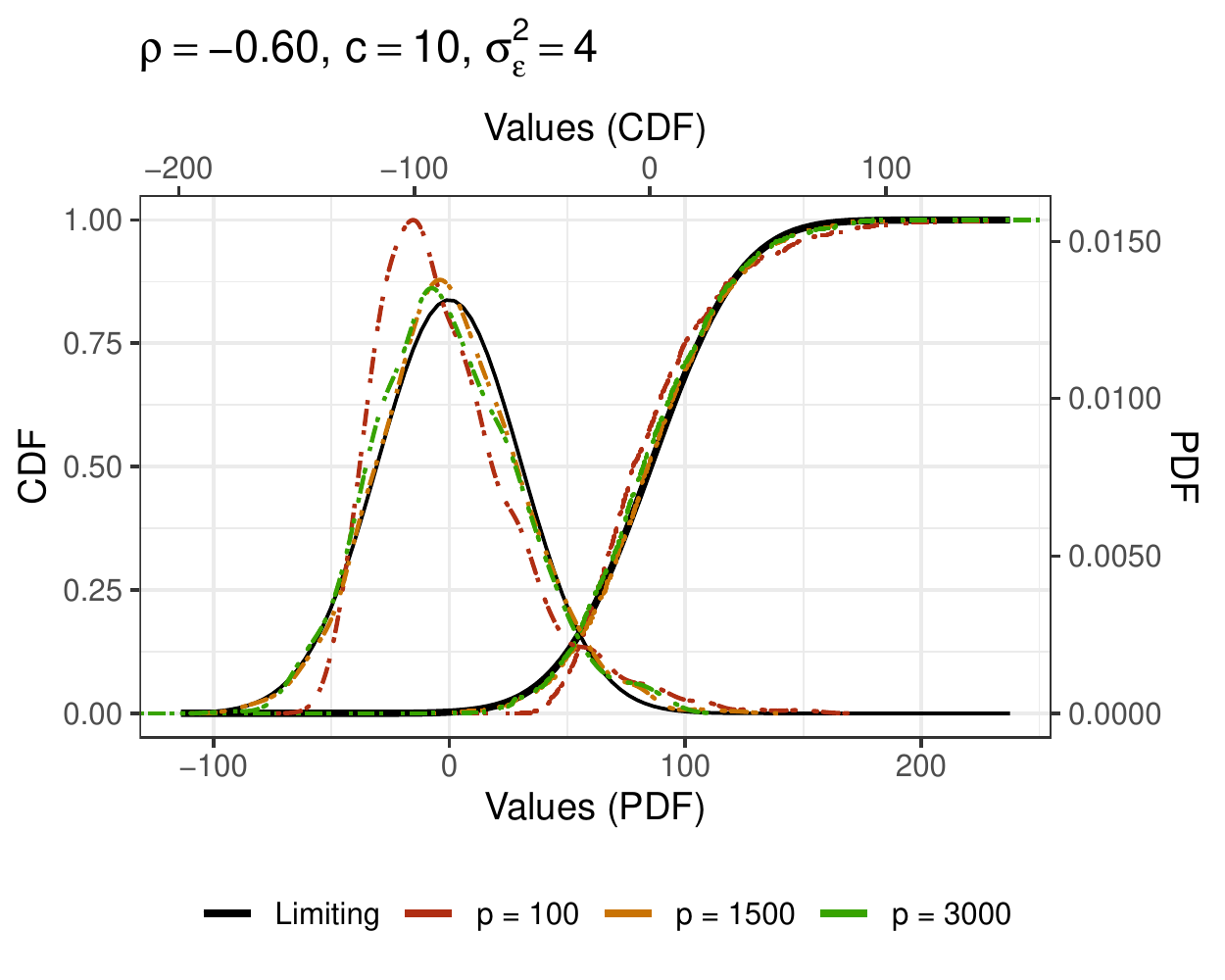}
	\caption{Comparison of the PDF and CDF using 1000 replications from the Monte Carlo simulation of the statistic \eqref{eq:th_sparse} with the limiting distribution $\mathcal{N}(0,s^2)$ by the Corollary \ref{cor:main} (in black) for $\rho = -0.6$, $c=1$, $\sigma_{\varepsilon}^2 = 4$ (left) and $\rho = -0.6$, $c=10$, $\sigma_{\varepsilon}^2 = 4$ (right). }
	\label{fig:3}
\end{figure}
\begin{figure}
	\centering
	\includegraphics[width=0.49\linewidth]{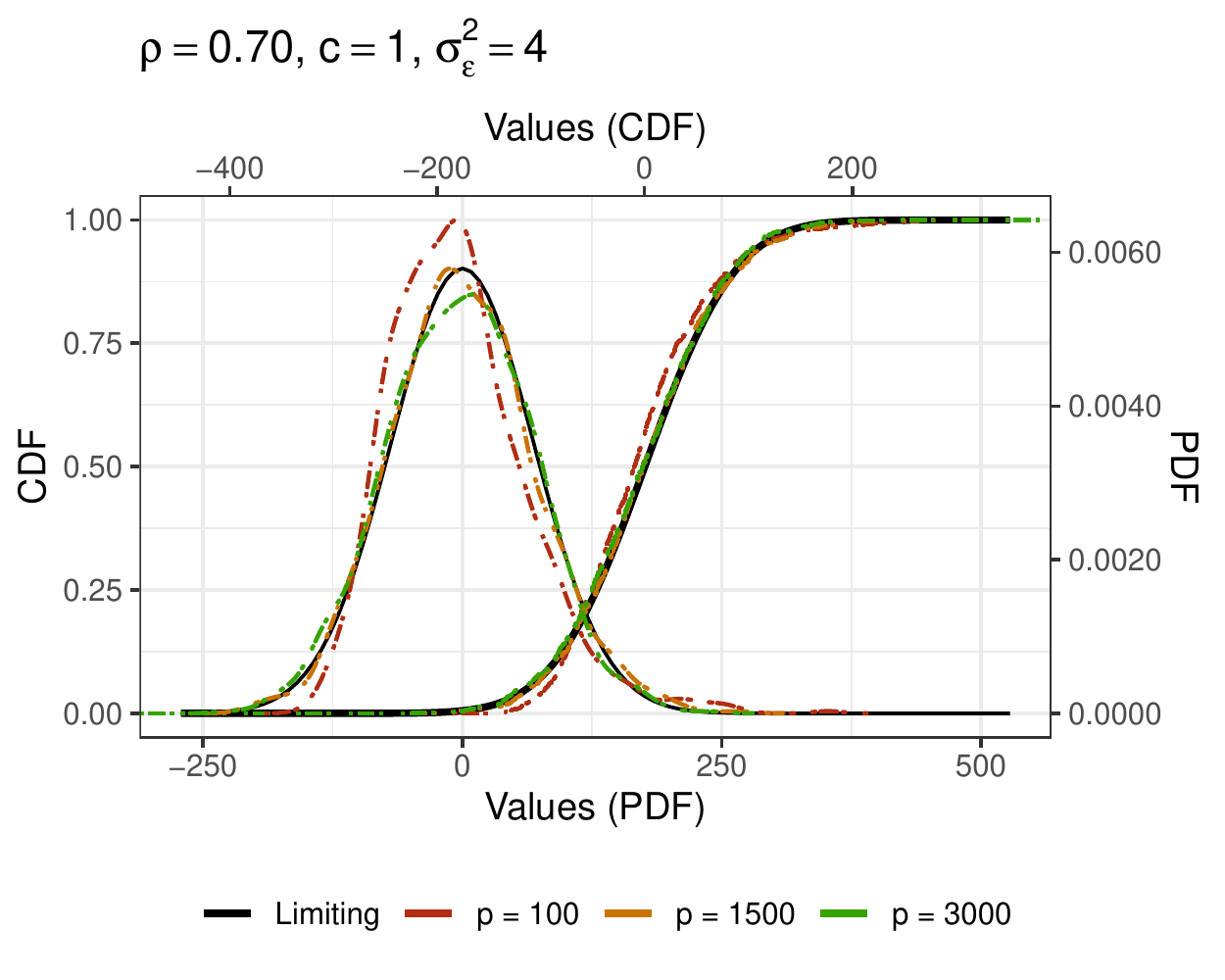} \includegraphics[width=0.49\linewidth]{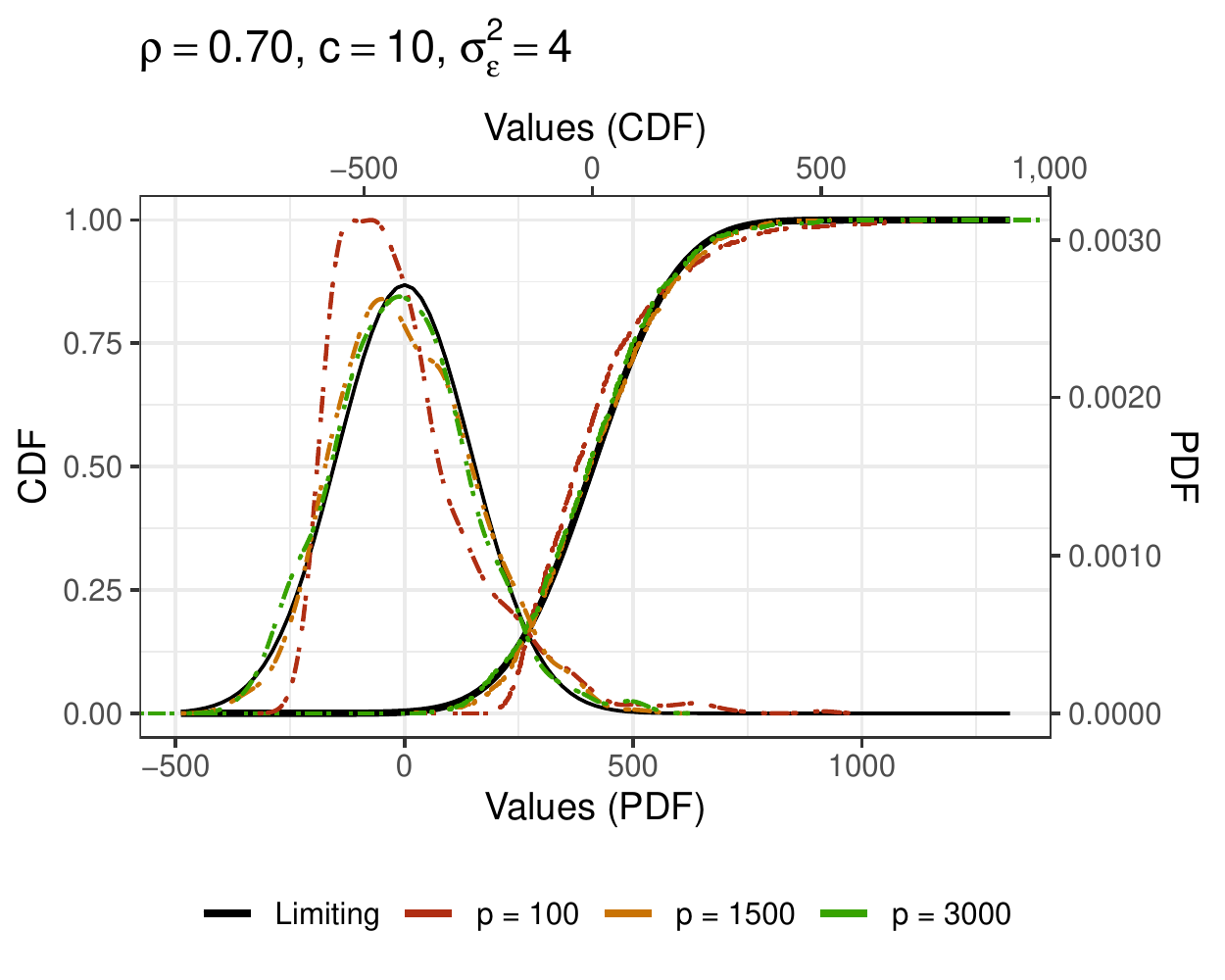}
	\caption{Comparison of the PDF and CDF using 1000 replications from the Monte Carlo simulation of the statistic \eqref{eq:th_sparse} with the limiting distribution $\mathcal{N}(0,s^2)$ by the Corollary \ref{cor:main} (in black) for $\rho = 0.7$, $c=1$, $\sigma_{\varepsilon}^2 = 4$ (left) and $\rho =0.7$, $c=10$, $\sigma_{\varepsilon}^2 = 4$ (right). }
	\label{fig:4}
\end{figure}
\begin{figure}
	\centering
	\includegraphics[width=0.49\linewidth]{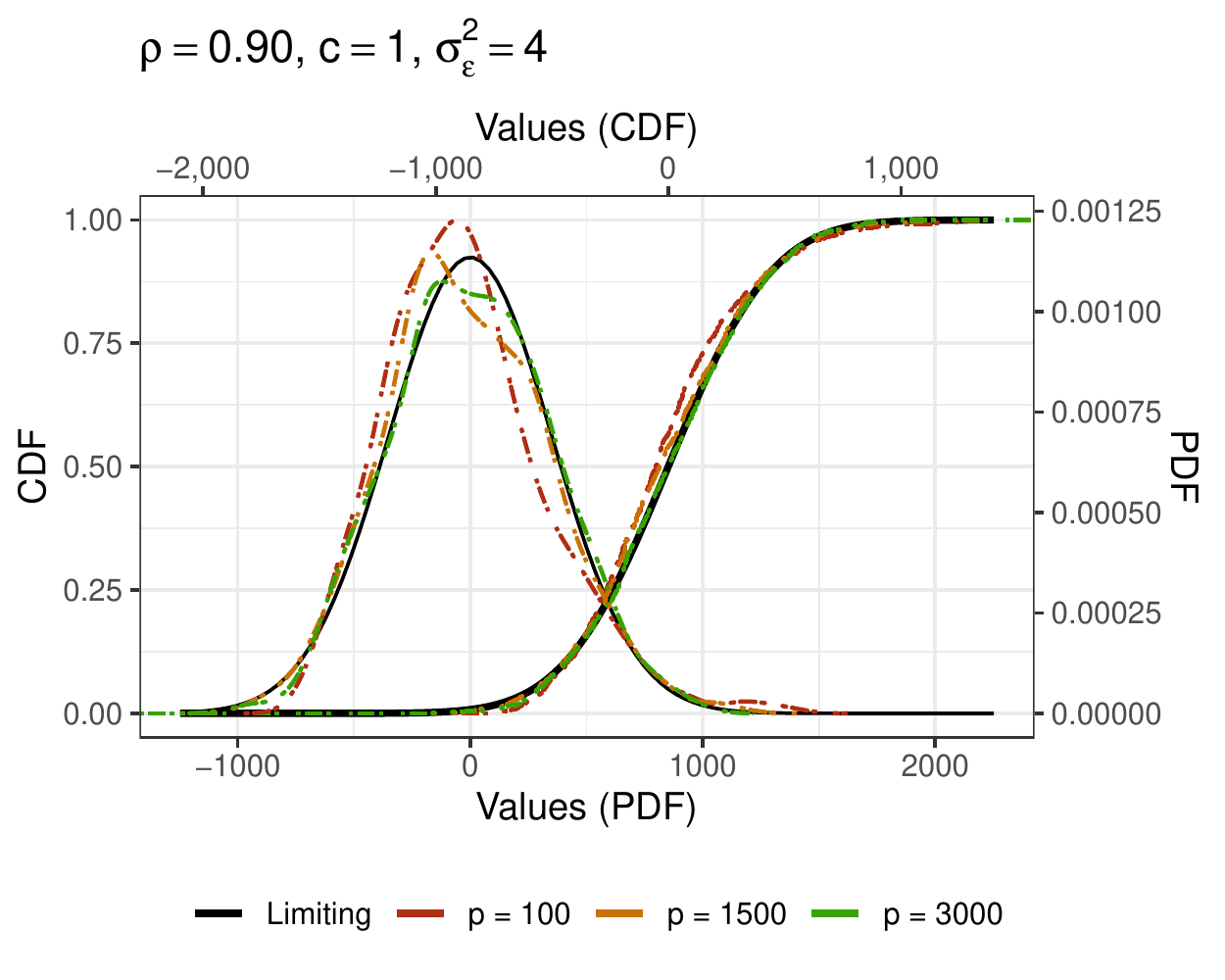} \includegraphics[width=0.49\linewidth]{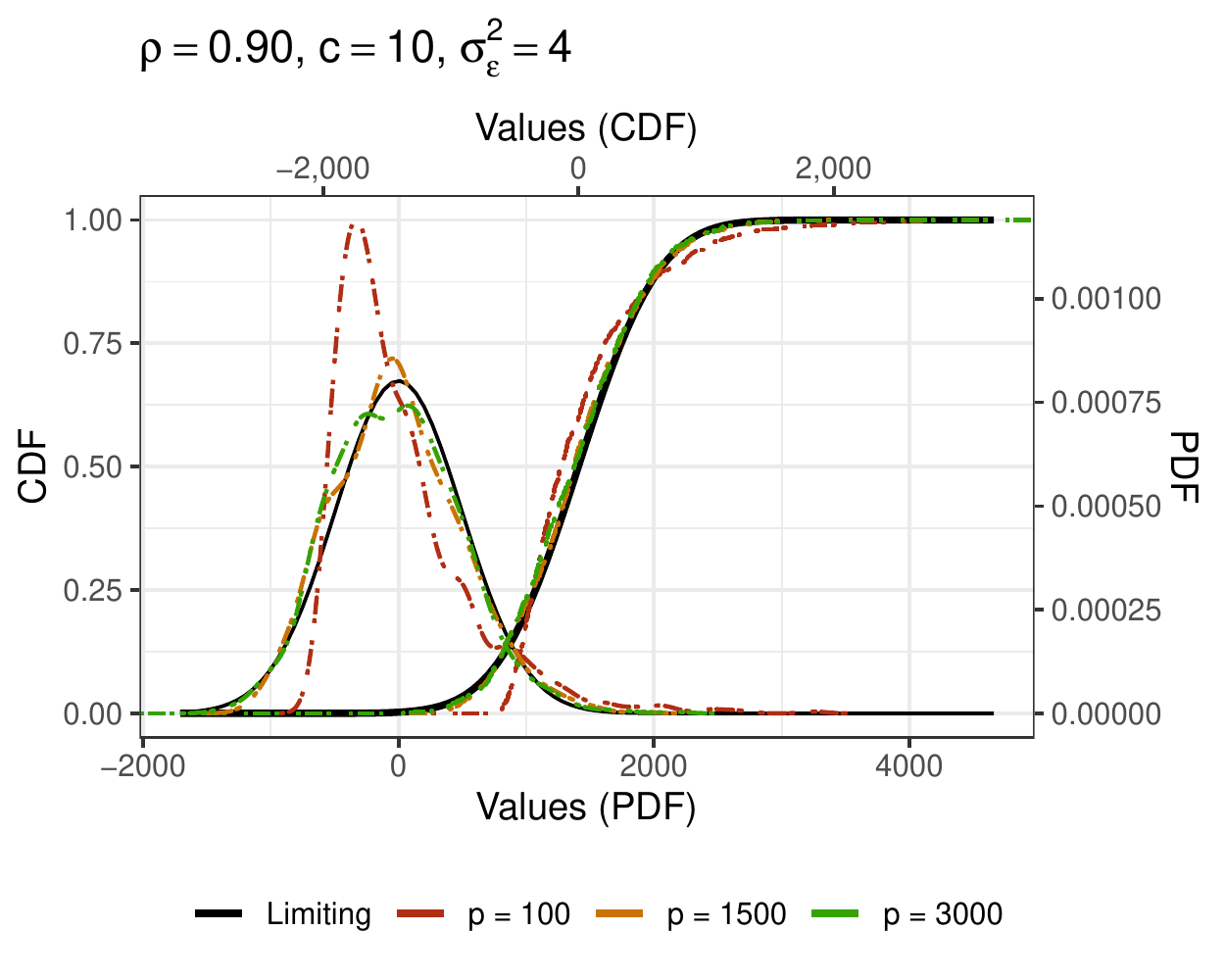}
	\caption{Comparison of the PDF and CDF using 1000 replications from the Monte Carlo simulation of the statistic \eqref{eq:th_sparse} with the limiting distribution $\mathcal{N}(0,s^2)$ by the Corollary \ref{cor:main} (in black) for $\rho = 0.9$, $c=1$, $\sigma_{\varepsilon}^2 = 4$ (left) and $\rho = 0.9$, $c=10$, $\sigma_{\varepsilon}^2 = 4$ (right). }
	\label{fig:5}
\end{figure}
\begin{figure}
	\centering
	\includegraphics[width=0.49\linewidth]{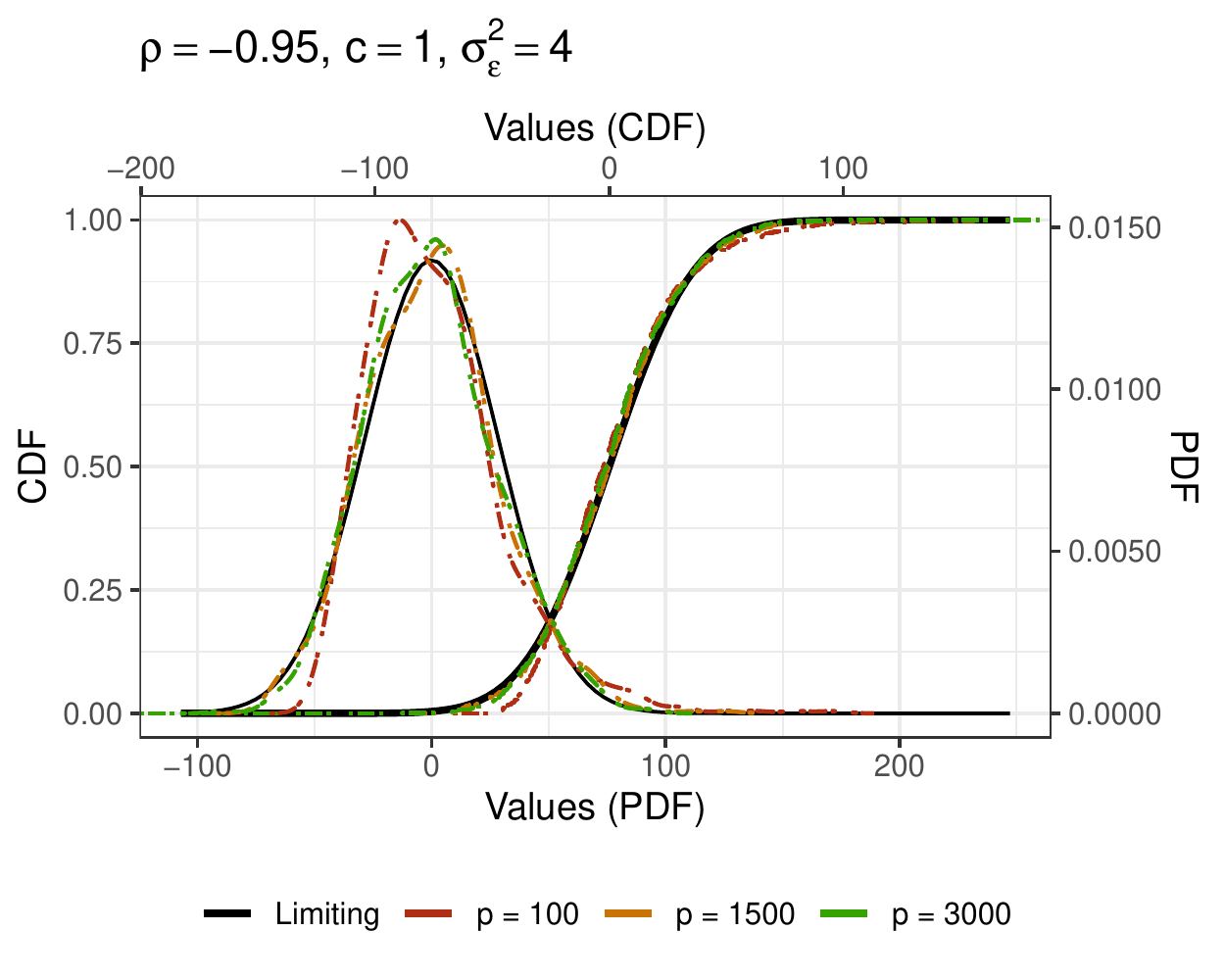} \includegraphics[width=0.49\linewidth]{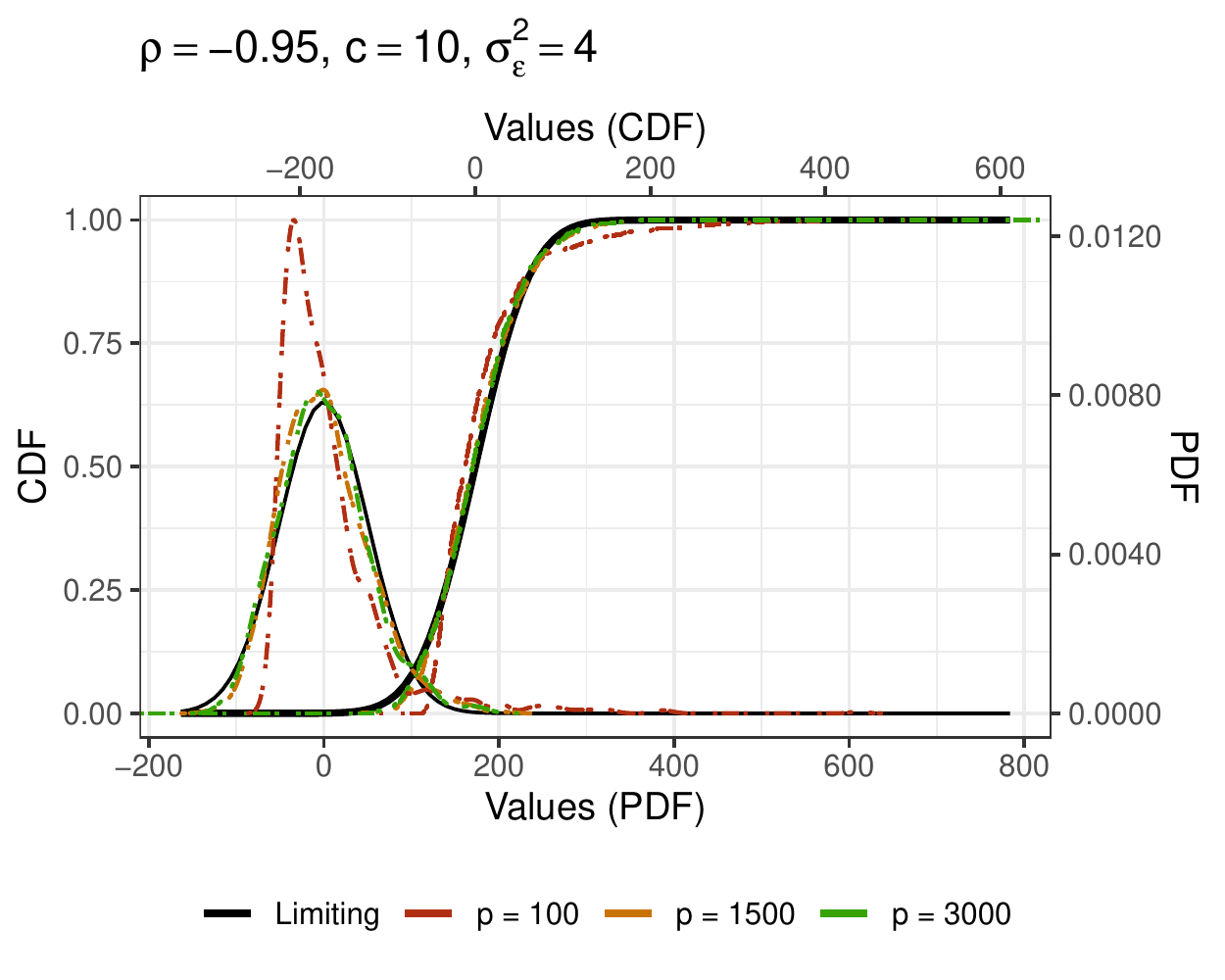}
	\caption{Comparison of the PDF and CDF using 1000 replications from the Monte Carlo simulation of the statistic \eqref{eq:th_sparse} with the limiting distribution $\mathcal{N}(0,s^2)$ by the Corollary \ref{cor:main} (in black) for $\rho = -0.95$, $c=1$, $\sigma_{\varepsilon}^2 = 4$ (left) and $\rho = -0.95$, $c=10$, $\sigma_{\varepsilon}^2 = 4$ (right). }
	\label{fig:6}
\end{figure}
\clearpage
\vspace{6pt}
\appendix

\section{Appendix}

\noindent
Throughout the proofs we use the notation $C$ to mark generic constants, the specific values of which can change from line to line.
\subsection{Technical lemmas}
\begin{Lemma}\label{lemma:integral2}
	Assume that $|\rho|<1$.
	Then,
	\begin{eqnarray*}
		\int_{0}^{\rho} \frac{\log (1-\rho x)}{x(1-x)} \dd x &=&  -\frac{1}{2}\big(\operatorname{Li}_{2}(\rho^2) + \log^2(1-\rho) \big),
	\end{eqnarray*}
	where $\operatorname{Li}_2$ denotes the real dilogarithm function. (Recall, that for $\rho < 0$, by $\int^{\rho}_0$ we denote $-\int_{\rho}^{0}$.)
\end{Lemma}
\begin{proof}
	Write,
	\begin{eqnarray}
		\int_{0}^{\rho}\frac{\log(1-\rho x )}{x(1-x)} \dd x
		&=& \int_{0}^{\rho}\frac{\log(1-\rho x )}{x} \dd x ~+~ \int_{0}^{\rho}\frac{\log(1-\rho x )}{1-x} \dd x \notag.
	\end{eqnarray}
	By \eqref{eq:polylog}, we have
	\begin{eqnarray}
		\int_{0}^{\rho}\frac{\log(1-\rho x )}{x} \dd x &=& - \operatorname{Li}_{2}(\rho^2). \label{eq:integral_1}
	\end{eqnarray}
	It remains to show that
	\begin{eqnarray}
		\int_0^{\rho}\frac{\log(1-\rho x )}{1-x} \dd x &=& \frac{1}{2}\big(\operatorname{Li}_{2}(\rho^2) - \log^2(1-\rho)\big). \label{eq:integral_1a}
	\end{eqnarray}
	Indeed, by substitution $v = \rho - \rho x $, we have
	\begin{eqnarray}\label{eq:integral_2}
		\int_{0}^{\rho}\frac{\log(1-\rho x )}{1-x} \dd x
		&=& \int_{\rho - \rho^2}^{\rho}\frac{\log(1-\rho + v )}{v} \dd v \notag \\
		&=&\int_{\rho - \rho^2}^{\rho}\frac{\log( 1 + \frac{v}{1-\rho} )}{v} \dd v  ~-~ \log^2(1-\rho).
	\end{eqnarray}
	Further, by substitution $w = -\frac{v}{1-\rho}$, we have
	\begin{eqnarray}
		\int_{\rho - \rho^2}^{\rho}\frac{\log( 1 + \frac{v}{1-\rho} )}{v} \dd v
		&=& - \int_{-\frac{\rho}{1-\rho}}^{-\rho}\frac{\log( 1 - w )}{w} \dd w \notag \\
		&=& \operatorname{Li}_{2}(-\rho) - \operatorname{Li}_{2}\Big(-\frac{\rho}{1-\rho}\Big) \notag \\
		&=& \operatorname{Li}_{2}(-\rho)  + \operatorname{Li}_{2}(\rho) + \frac{1}{2}\log^2(1-\rho) \label{eq:maximon_1}\\
		&=& \frac{1}{2}\big(\operatorname{Li}_{2}(\rho^2) + \log^2(1-\rho) \big), \label{eq:maximon_2}
	\end{eqnarray}
	where for \eqref{eq:maximon_1}--\eqref{eq:maximon_2} we apply the easily verifiable identities (see, e.g., \cite{maximon2003dilogarithm}):
	\begin{eqnarray*}
		\operatorname{Li}_{2}\Big(\frac{x}{x-1}\Big) &=& -\operatorname{Li}_{2}(x) - \frac{1}{2}\log^2(1-x), ~~~ x < 1, \\
		\operatorname{Li}_{2}(x) + \operatorname{Li}_{2}(-x) &=& \frac{1}{2}\operatorname{Li}_{2}(x^2),~~~ |x|<1.
	\end{eqnarray*}
	Thus, \eqref{eq:integral_2} and \eqref{eq:maximon_2} imply \eqref{eq:integral_1a},
	which concludes the proof.
\end{proof}

\begin{Lemma} \label{lemma:ineqs}
	Assume that $\sum_{j=1}^{\infty}\beta_j^2 < \infty$ and $|\rho|<1$.  Then, the following inequalities hold:
	\begin{enumerate}[label={(\rm \roman*)}]
		\item $\begin{aligned}
			\bigg|\sum_{l=p+1}^{\infty}\sum_{l'=l+1}^{\infty} \beta_l \beta_{l'} \rho^{l'-l}\bigg| ~~\leq~~  C \sum_{l=p+1}^{\infty} \beta_l^2.  \\ 		
		\end{aligned}$
		\item $\begin{aligned}[t]
			\bigg|\sum_{l=p+1}^{\infty}\sum_{l'=l+1}^{\infty}\beta_l \beta_{l'}\rho^{l' -l } (l'-l)\bigg| ~~\leq~~  C \sum_{l=p+1}^{\infty}\beta_{l}^2.\\
		\end{aligned}$
		\item $\begin{aligned}
			\bigg|\sum_{l=1}^{p} \sum_{l'=p+1}^{\infty }\beta_{l} \beta_{l'} \rho^{l'-l}\bigg| ~~\leq~~  C \sum_{l=p+1}^{\infty} \beta_{l}^{2}.
		\end{aligned}$
		\item $\begin{aligned}
			\bigg|\sum_{l=1}^{p} \sum_{l'=p+1}^{\infty }\beta_{l} \beta_{l'} \rho^{l'+l}\bigg| ~~\leq~~  C \sum_{l=p+1}^{\infty} \beta_{l}^{2}.
		\end{aligned}$
	\end{enumerate}
	
\end{Lemma}
\begin{proof}
	See the proof in Supplementary material, Section~\ref{sec:supp_ineqs}.
\end{proof}

\begin{Lemma} \label{lemma:bj1}
	Assume that $\sup_{j \geq 1} |\beta_j| j^{\alpha} < \infty$, $\alpha > 1/2$ and that $|\rho| < 1$. Then,
	\begin{align*}
		\bigg|\sum_{j=1}^{p}\beta_{j} \rho^{p-j} \bigg|  = o(p^{-1/4}).
	\end{align*}
\end{Lemma}
\begin{proof}
	We have
	\begin{eqnarray}
		\bigg| \sum_{j=1}^{p}\beta_j \rho^{p-j} \bigg| &\leq&   \sum_{j=1}^{\lfloor \sqrt{p} \rfloor} |\beta_j| |\rho|^{p-j} + \sum_{j=\lfloor \sqrt{p} \rfloor + 1}^{p} |\beta_j| |\rho|^{p-j}  \notag \\
		&\leq&  \sup_{j \geq 1}|\beta_j| \sum_{j=1}^{\lfloor \sqrt{p} \rfloor}|\rho|^{p-j} + p^{-\alpha/2} \sum_{j = \lfloor \sqrt{p} \rfloor + 1}^{p} |\beta_j| p^{\alpha/2} |\rho|^{p-j} \notag \\
		&\leq&  \sup_{j \geq 1}|\beta_j| \sum_{j=1}^{\lfloor \sqrt{p} \rfloor}|\rho|^{p-j} + p^{-\alpha/2} \sup_{j \geq 1} |\beta_j| j^{\alpha} \sum_{j = \lfloor \sqrt{p} \rfloor + 1}^{p} |\rho|^{p-j} \notag \\
		&\leq&  C\bigg( \sum_{j=1}^{\lfloor \sqrt{p} \rfloor}|\rho|^{p-j} + p^{-\alpha/2} \sum_{j=\lfloor \sqrt{p} \rfloor + 1}^{p} |\rho|^{p-j} \bigg)  \notag \\
		&\leq&  C\Big(|\rho|^{p - \lfloor \sqrt{p} \rfloor}  + p^{-\alpha/2} \Big). \notag
	\end{eqnarray}
	Here we used the fact that $\sum_{j=\lfloor \sqrt{p} \rfloor + 1}^{p} |\rho|^{p-j} \to (1-|\rho|)^{-1}<\infty$.
	Thus,
	\begin{eqnarray}
		p^{1/4} \bigg| \sum_{j=1}^{p}\beta_j \rho^{p-j} \bigg|\leq C \left(  p^{1/4}|\rho|^{p - \lfloor \sqrt{p} \rfloor}  + p^{\frac{1}{4}-\frac{\alpha}{2}} \right) ~\to~ 0.
	\end{eqnarray}
\end{proof}
\begin{remark} Obviously, the assumption $\sup_{j\geq 1} |\beta_j| j^{\alpha} < \infty$, for $\alpha > 1/2$, implies that \mbox{$\sum_{j=1}^{\infty}\beta_j^2 < \infty$:}
	\begin{eqnarray*}
		\sum_{j=1}^{\infty}\beta_j^2 &=& \sum_{j=1}^{\infty}\beta_j^2 j^{2 \alpha }j^{-2\alpha} ~~\leq~~ \sup_{j\geq 1} \beta_j^2 j^{2\alpha} \sum_{k=1}^{\infty} k^{-2\alpha}   ~~<~~ \infty.
	\end{eqnarray*}
\end{remark}

\begin{Lemma} \label{lemma:kappa_2p_op}
	Assume that the assumptions of Theorem \ref{th:wider} hold. Then,
	\begin{eqnarray*}
		\kappa_{2,p} = o(p).
	\end{eqnarray*}
\end{Lemma}
\begin{proof}
	Observe, that
	\begin{eqnarray}
		\kappa_{2,p} ~=~ \sum_{k=1}^p \bigg(\sum_{l=1}^p \beta_l \rho^{|k-l|}\bigg)^2
		&=& \sum_{k=1}^p \sum_{l_1,l_2=1}^p \beta_{l_1}\beta_{l_2} \rho^{|k-l_1|+|k-l_2|}\nonumber\\
		&\le &  \sum_{l_1,l_2=1}^p |\beta_{l_1}| |\beta_{l_2}| \sum_{k=1}^p |\rho|^{|k-l_1|+|k-l_2|}\nonumber\\
		&\le &  C \bigg(\sum_{l=1}^p |\beta_{l_1}|\bigg)^2\label{ub} \\
		&=&  o(p) \notag
	\end{eqnarray}
	where \eqref{ub} follows from \eqref{uii}. Meanwhile,
	$
	\sum_{l=1}^p |\beta_{l_1}| = o(p^{1/2})$, since
	\begin{eqnarray*}
		\sum_{l=1}^p |\beta_l| &=& \sum_{l=1}^{\lfloor p^{1/2}\rfloor} |\beta_l| ~+~
		\sum_{l=\lfloor p^{1/2}\rfloor+1}^p |\beta_l| \\
		&\le& p^{1/4} \bigg(\sum_{l=1}^\infty \beta^2_l\bigg)^{1/2}~+~ p^{1/2} \bigg(\sum_{l=\lfloor p^{1/2}\rfloor+1}^\infty \beta_l^2\bigg)^{1/2} \ = \ o(p^{1/2}).
	\end{eqnarray*}
\end{proof}

\begin{Lemma} \label{lemma:lambda_3}
	Assume that $\sum_{j=1}^{\infty} \beta_j^2 < \infty$ and $|\rho|<1$. Define $\theta_{k}^{(p)} = \sum_{j=1}^{p} \beta_j \rho^{|k-j|}$. Then,
	\begin{eqnarray} \label{eq:lambda_3_3}
		\bigg|	\sum_{i,j,k=1}^{p}\big(\rho^{|i-j|} + \theta_{i}^{(p)} \theta_{j}^{(p)}  \big)\big(\rho^{|i-k|} + \theta_{i}^{(p)} \theta_{k}^{(p)}  \big)\big(\rho^{|k-j|} + \theta_{k}^{(p)} \theta_{j}^{(p)}  \big) \bigg|  &=& o(p^{3/2}).~~~~~~~~~~
	\end{eqnarray}
\end{Lemma}
\begin{proof}
	See the proof in Supplementary material, Section~\ref{sec:supp_lambda_3}.
\end{proof}

\subsection{Proof of Lemma \ref{lemma:alt_expressions} }
\label{ap:lemma45}

\noindent
Here and throughout the proof we employ the notation as in Definition \ref{def:main}.

\noindent {\rm (i)}
Note that, by (\ref{beta}) and (\ref{gamma}), we have
\begin{eqnarray*}
	\kappa_{1,p} 		&=& \sum_{k=1}^p
	\beta_k^2 + 2\sum_{k=2}^p\sum_{l=1}^{k-1}\beta_k \beta_l \rho^{k-l}  \ \to\ \beta(1) + 2b_{1}(\rho)  ~~\text{ as } \ p \to \infty.
\end{eqnarray*}

\noindent{\rm (ii)}
Write
\begin{eqnarray*}
	\kappa_{2,p}   &=&  \sum_{l=1}^{p}\sum_{k=1}^{p}\beta_l^2 \rho^{2|k-l|}+ 2\sum_{l' > l }\sum_{k=1}^{p}\beta_l \beta_{l'} \rho^{|k-l|}\rho^{|k-l'|}.
\end{eqnarray*}
From here, it is straightforward to see that
\begin{eqnarray}
	\kappa_{2,p} &\to& \beta(1) \frac{1+\rho^2}{1-\rho^2} -\beta(\rho^2)\frac{1}{1-\rho^2} + 2\Big( b_{1}^{(1)}(\rho) ~+~ b_{1}(\rho)\frac{1+\rho^2}{1-\rho^2} ~-~ b_{2}(\rho)\frac{1}{1-\rho^2} \Big).~~~~~~~ \label{eq:supp_1}
\end{eqnarray}
Technical details of the proof of \eqref{eq:supp_1} are presented in Supplementary material, Section~\ref{sec:supp_a2}. 
\smallskip

\noindent	{\rm (iii)}
Consider
\begin{eqnarray} \label{eq:part3_1}
	\kappa_{3,p} &=& \sum_{l=1}^p \beta_l^2 J_{1}(l) + 2 \sum_{l<l'}\beta_{l} \beta_{l'}  J_2(l,l'),~~~~~~~~~~
\end{eqnarray}
where
\begin{eqnarray}
	J_{1}(l) &:=&  \sum_{k,k'=1}^p\rho^{|k-k'|} \rho^{|k-l|} \rho^{|k'-l|} {\bf 1}_{\{l=l'\}}, \label{J1l} \\
	J_{2}(l,l') &:=&  \sum_{k,k'=1}^p\rho^{|k-k'|} \rho^{|k-l|} \rho^{|k'-l'|} {\bf 1}_{\{l<l'\}}. \label{J2ll}
\end{eqnarray}
Then, it is straightforward to see that, as $p \to \infty$, using the notation in Definition~\ref{def:main}, we have that
\begin{eqnarray}
	\sum_{l=1}^{p}\beta_{l}^{2} J_{1}(l) &\to& \beta(1) \frac{1+4 \rho^2+\rho^4}{(1-\rho^2)^2}- \beta(\rho^2) \frac{1+3 \rho^2}{(1-\rho^2)^2}-\frac{2}{1 - \rho^2} \beta^{(1)}(\rho^2), \label{eq:supp_j1}
\end{eqnarray}
and
\begin{eqnarray}
	\sum_{l'>l}\beta_{l}^2 J_{2}(l,l')
	&\to& \frac{1}{2(1-\rho^2)^2}\big( b^{(2)}(\rho) (1-\rho^2)^2
	+ 3b_{1}^{(1)}(\rho) (1-\rho^4) + 2b_{1}(\rho)(1 + 4\rho^2 + \rho^4) \notag\\
	&&-~ 2b_{2}^{(1)}(\rho)(1-\rho^2)- 2 b_{2}(\rho)(1+3\rho^2) \big). \label{eq:supp_j2}
\end{eqnarray}
Technical details of the proof of \eqref{eq:supp_j1}--\eqref{eq:supp_j2} are omitted here and presented in the Supplementary material~\ref{sec:supp_a3}. This concludes the proof.


\newpage

\section*{Supplementary material}

\beginsupplement

\subsection{Proof of Lemma \ref{lemma:op}} \label{sec:B1}
\label{sec:supp_op}

\noindent\underline{Part (i)}. Write
\begin{eqnarray*}
	\kappa_1-\kappa_{1,p} &=& 2
	\sum_{l=1}^{p}\sum_{l'=p+1}^{\infty}\beta_{l} \beta_{l'} \rho^{l'-l} +
	\sum_{l=p+1}^\infty\sum_{l'=p+1}^{\infty}\beta_{l} \beta_{l'} \rho^{|l-l'|}.
\end{eqnarray*}
Here,
\begin{eqnarray*}
	\bigg|\sum_{l=1}^{p}\sum_{l'=p+1}^{\infty}\beta_{l} \beta_{l'} \rho^{l'-l}\bigg| &=&  o(p^{-1/2})
\end{eqnarray*}
by Lemma \ref{lemma:ineqs}(iii), while
\begin{eqnarray*}
	\sum_{l=p+1}^{\infty}\sum_{l'=p+1}^{\infty}\beta_{l} \beta_{l'} \rho^{|l-l'|}
	&=& \sum_{l=p+1}^{\infty}\beta_l^2 ~+~ 2 \sum_{l=p+1}^{\infty}\sum_{l'=l+1}^{\infty} \beta_l \beta_{l'}\rho^{l'-l} \ \leq\
	C  \sum_{l=p+1}^{\infty} \beta_l^2 ~~=~~ o(p^{-1/2})
\end{eqnarray*}
by Lemma \ref{lemma:ineqs}(i), which concludes the proof.\hfill $\Box$\smallskip

\noindent \underline{Part (ii)}. Write
\begin{eqnarray}
\kappa_2-\kappa_{2,p} &=& \bigg(2 \sum_{k=1}^p\sum_{l=1}^p \sum_{l'=p+1}^\infty
+2 \sum_{k=p+1}^\infty \sum_{l=1}^p \sum_{l'=p+1}^\infty
+\sum_{k=1}^p \sum_{l=p+1}^\infty \sum_{l'=p+1}^\infty +\ \sum_{k=p+1}^p\sum_{l=1}^p \sum_{l'=1}^p \notag\\
&& +\ \sum_{k=p+1}^p\sum_{l=p+1}^\infty \sum_{l'=p+1}^\infty\bigg) \beta_l\beta_{l'}\rho^{|k-l|+|k-l'|}\notag\\
&=:& 2L_1+2L_2+L_3+L_4+L_5. \label{LLLLL}
\end{eqnarray}
We have
\begin{eqnarray*}
L_1 			 &=& \sum_{l=1}^{p} \sum_{l'=p+1}^{\infty}  \beta_{l} \beta_{l'} \sum_{k=1}^{p}\rho^{|k-l|+l'-k} \ =\  \sum_{l=1}^{p} \sum_{l'=p+1}^{\infty} \beta_{l} \beta_{l'}\bigg(\sum_{k=1}^{l} \rho^{l'+l-2k}  + \sum_{k=l+1}^{p}  \rho^{l'-l} \bigg) \\
&=& \sum_{l=1}^{p} \sum_{l'=p+1}^{\infty} \beta_{l} \beta_{l'}\Big(\frac{\rho^{l'-l}}{1-\rho^2} - \frac{\rho^{l+l'}}{1-\rho^2}  +(p-l)\rho^{l'-l} \Big).
\end{eqnarray*}
Thus,
\begin{eqnarray*}
|L_1|			&\leq&  \sum_{l=1}^{p} \sum_{l'=p+1}^{\infty} |\beta_{l}| |\beta_{l'}|\Big(\frac{|\rho|^{l'-l}}{1-\rho^2} + \frac{|\rho|^{l+l'}}{1-\rho^2}  +(l'-l)|\rho|^{l'-l} \Big) \ \le\ C \sum_{j=p+1}^{\infty}\beta_{j}^{2} = o(p^{-1/2})
\end{eqnarray*}
by Lemma \ref{lemma:ineqs}(ii)--(iv). For the term $L_2$ we have
\begin{eqnarray*}
L_2 &=&  \sum_{l=1}^{p} \sum_{l'=p+1}^{\infty}\beta_{l} \beta_{l'} \sum_{k=p+1}^{\infty} \rho^{k-l+|k-l'|} \ =\  \sum_{l=1}^{p} \sum_{l'=p+1}^{\infty}\beta_{l} \beta_{l'} \bigg( \sum_{k=p+1}^{l'} \rho^{l'-l}  + \sum_{k=l'+1}^{\infty} \rho^{2k-l-l'}\bigg).
\end{eqnarray*}
Thus, by Lemma \ref{lemma:ineqs}(ii)-(iii),
\begin{eqnarray*}
|L_2|  &\leq&  \sum_{l=1}^{p} \sum_{l'=p+1}^{\infty}|\beta_{l}| |\beta_{l'}| \Big( (l'-l)|\rho|^{l'-l}  + \frac{\rho^2}{1-\rho^2}|\rho|^{l'-l}\Big) \ =\  o(p^{-1/2}).
\end{eqnarray*}
For the term $L_3$  we have
\begin{eqnarray*}
L_3 &=& \sum_{l=p+1}^{\infty}\sum_{l'=p+1}^{\infty}\beta_{l} \beta_{l'} \sum_{k=1}^{p}\rho^{l+l'-2k} \ =\
\sum_{l=p+1}^{\infty}\sum_{l'=p+1}^{\infty}\beta_{l} \beta_{l'} \bigg( \frac{\rho^{l'+l-2p}}{1-\rho^2} - \frac{\rho^{l+l'}}{1-\rho^2} \bigg).
\end{eqnarray*}
Thus,
\begin{eqnarray*}
|L_3|	&\leq& \sum_{l=p+1}^{\infty}\sum_{l'=p+1}^{\infty}|\beta_{l}| |\beta_{l'}| \Big( \frac{|\rho|^{l'+l-2p}}{1-\rho^2} + \frac{|\rho|^{l+l'}}{1-\rho^2} \Big) \\
& \leq& \frac{1}{1-\rho^2}\bigg(\bigg( \sum_{l=p+1}^{\infty} |\beta_{l}||\rho|^{l-p}\bigg)^2 + \bigg( \sum_{l=p+1}^{\infty}|\beta_{l}| |\rho|^{l} \bigg)^2 \bigg) \\
& \leq&  C \sum_{l=p+1}^{\infty} \beta_{l}^2 + o(p^{-1/2})  = o(p^{-1/2}),
\end{eqnarray*}
since $\sum_{l=p+1}^{\infty} |\rho|^{2(l-p)} < \infty$ and due to H\"older's inequality. Further, for the term $L_4$ we have
\begin{eqnarray*}
L_4 &=&  \sum_{l=1}^{p} \sum_{l'=1}^{p} \beta_{l} \beta_{l'}  \sum_{k=p+1}^{\infty}\rho^{2k-l-l'} \ =\  \frac{\rho^2}{1-\rho^2}\sum_{l=1}^{p} \sum_{l'=1}^{p} \beta_{l} \beta_{l'}  \rho^{(p-l)+(p-l)}\\
&=&  \frac{\rho^2}{1-\rho^2}\bigg( \sum_{l=1}^{p} \beta_{l} \rho^{p-l}\bigg)^2 \ =\ o(p^{-1/2}),
\end{eqnarray*}
by Lemma \ref{lemma:bj1}.  Finally, for $L_5$ write
\begin{eqnarray*}
L_5 &=& \sum_{k=p+1}^{\infty} \sum_{l=p+1}^{\infty}\beta_l^2 \rho^{2|k-l|} + 2\sum_{k=p+1}^{\infty} \sum_{l,l'=p+1,\,l'>l}^{\infty}\beta_l \beta_{l'}\rho^{|k-l|+|k-l'|}.
\end{eqnarray*}
For the first summand,  we have
\begin{eqnarray}
\sum_{k,l=p+1}^{\infty} \beta_{l}^{2}\rho^{2|k-l|}  &=& \sum_{l=p+1}^{\infty}\beta_l^2 \sum_{k=p+1}^{\infty}   \rho^{2|k-l|} \mathbf{1}_{\{k \geq l\}} + \sum_{l=p+1}^{\infty}\beta_l^2 \sum_{k=p+1}^{\infty}   \rho^{2|k-l|} \mathbf{1}_{\{k < l\}} \notag \\
&=& \sum_{l=p+1}^{\infty}\beta_l^2 \sum_{k=l}^{\infty}   \rho^{2k-2l}  + \sum_{l=p+1}^{\infty}\beta_l^2 \sum_{k=p+1}^{l+1}   \rho^{2l-2k} \notag \\
&=& \sum_{l=p+1}^{\infty}\beta_l^2\bigg( \Big(\frac{1}{1-\rho^2}\Big)  + \Big(-\frac{\rho^{2 l-2 p}}{1-\rho^2}+\frac{1}{\rho^2 \left(1-\rho^2\right)}\Big) \bigg) \notag \\
&\leq& C \sum_{l=p+1}^{\infty}\beta_l^2,\label{eq:ineq:0}
\end{eqnarray}
where $C < \infty$. Similarly,
\begin{eqnarray}
\sum_{k=p+1}^{\infty }\sum_{l'>l}^{\infty}\beta_l \beta_{l'}\rho^{|k-l|+|k-l'|} 
&  =& \sum_{l=p+1}^{\infty}\sum_{l'=l+1}^{\infty}\sum_{k=p+1}^{\infty}\beta_l \beta_{l'}\rho^{2k-l-l'}\mathbf{1}_{\{k\geq l' > l\}} \label{eq:ineq:1} \\
&&+~~ \sum_{l=p+1}^{\infty}\sum_{l'=l+1}^{\infty}\sum_{k=p+1}^{\infty}\beta_l \beta_{l'}\rho^{l' -l }\mathbf{1}_{\{l' > k \geq l\}} \label{eq:ineq:2}\\
&& +~~ \sum_{l=p+1}^{\infty}\sum_{l'=l+1}^{\infty}\sum_{k=p+1}^{\infty}\beta_l \beta_{l'}\rho^{l'+l - 2k}\mathbf{1}_{\{l' > l > k\}}.~~~~~~~~~~~~~~ \label{eq:ineq:3}
\end{eqnarray}
For \eqref{eq:ineq:1}, write
\begin{eqnarray*}
\sum_{l=p+1}^{\infty}\sum_{l'=l+1}^{\infty}\sum_{k=p+1}^{\infty}\beta_l \beta_{l'}\rho^{2k-l-l'}\mathbf{1}_{\{k\geq l' > l\}} &=& \sum_{l=p+1}^{\infty}\sum_{l'=l+1}^{\infty} \beta_l \beta_{l'} \frac{\rho^{l'-l}}{1-\rho^2}.
\end{eqnarray*}
Thus, by Lemma~\ref{lemma:ineqs}(i), we have
\begin{eqnarray*}
\bigg|\sum_{l=p+1}^{\infty}\sum_{l'=l+1}^{\infty}\sum_{k=p+1}^{\infty}\beta_l \beta_{l'}\rho^{2k-l-l'}\mathbf{1}_{\{k\geq l' > l\}}\bigg| &\leq& C \sum_{l=p+1}^{\infty}\beta_{l}^2,
\end{eqnarray*}
for $C < \infty$. Next, for \eqref{eq:ineq:2}, write
\begin{eqnarray*}
\sum_{l=p+1}^{\infty}\sum_{l'=l+1}^{\infty}\sum_{k=p+1}^{\infty}\beta_l \beta_{l'}\rho^{l' -l }\mathbf{1}_{\{l' > k \geq l\}} &=& \sum_{l=p+1}^{\infty}\sum_{l'=l+1}^{\infty}\beta_l \beta_{l'}\rho^{l' -l }(l'-l).
\end{eqnarray*}
It follows from Lemma~\ref{lemma:ineqs}(ii) that
\begin{eqnarray*}
\bigg| \sum_{l=p+1}^{\infty}\sum_{l'=l+1}^{\infty}\sum_{k=p+1}^{\infty}\beta_l \beta_{l'}\rho^{l' -l }\mathbf{1}_{\{l' > k \geq l\}} \bigg|
&\leq& C\sum_{l=p+1}^{\infty}\beta_{l}^2,
\end{eqnarray*}
for $C < \infty$. Finally, for \eqref{eq:ineq:3}, we have
\begin{eqnarray*}
\sum_{k,l,l'=p+1}^{\infty}\beta_l \beta_{l'}\rho^{l'+l - 2k}\mathbf{1}_{\{l' > l > k\}} &=& \sum_{l=p+2}^{\infty}\sum_{l'=l+1}^{\infty}\beta_l \beta_{l'} \bigg(\frac{\rho^{(l-p) + (l'-p)}}{\rho^2-1}+\rho^2\frac{\rho^{l'-l}}{1-\rho^2}\bigg).
\end{eqnarray*}
Thus, by Lemma \ref{lemma:ineqs}(i),
\begin{eqnarray*}
\bigg|\sum_{k,l,l'=p+1}^{\infty}\beta_l \beta_{l'}\rho^{l'+l - 2k}\mathbf{1}_{\{l' > l > k\}}\bigg| &\leq& \frac{1}{2}\sum_{l=p+2}^{\infty}\sum_{l'=l+1}^{\infty}\left(\beta_l^2 \rho^{2(l'-p)} + \beta_{l'}^2 \rho^{2(l-p)}\right) ~+~ C\sum_{l=p+1}^{\infty}\beta_{l}^2 \\
&\leq&  C\sum_{l=p+1}^{\infty}\beta_{l}^2.
\end{eqnarray*}
Hence, \eqref{eq:ineq:0} and the estimates for
\eqref{eq:ineq:1}--\eqref{eq:ineq:3} yield
\begin{eqnarray*}
|L_5| &\leq& C \sum_{l=p+1}^{\infty} \beta_{l}^2 \ =\ o(p^{-1/2}).
\end{eqnarray*}
Equality \eqref{LLLLL} and estimates $|L_i|=o(p^{-1/2})$, $i=1,\dots,5$, complete the proof of
part (ii).\hfill ~$\Box$

\subsection{Proof of Lemma \ref{lemma:ineqs}}
\label{sec:supp_ineqs}
\begin{proof}
For inequality {\rm (i)}, we have that
\begin{eqnarray*}
	\bigg|\sum_{l=p+1}^{\infty}\sum_{l'=l+1}^{\infty} \beta_l \beta_{l'} \rho^{l'-l}\bigg| &\leq&
	\frac{1}{2} \sum_{l=p+1}^{\infty}\sum_{l'=l+1}^{\infty}\left( \beta_l^2|\rho|^{l'-l} + \beta_{l'}^2|\rho|^{l'-l} \right) \\
	&=&\frac{1}{2} \bigg( \frac{|\rho|}{1-|\rho|} \sum_{l=p+1}^{\infty} \beta_l^2  + \sum_{l'=p+1}^{\infty}\beta_{l'}^2 \bigg(-\frac{|\rho|^{l'-p}}{1-|\rho|}+\frac{1}{1-|\rho|}\bigg) \bigg) \\
	&\leq&  C \sum_{l=p+1}^{\infty} \beta_l^2.
\end{eqnarray*}
For inequality {\rm (ii)}, note that
\begin{eqnarray*}
	\bigg|\sum_{l=p+1}^{\infty}\sum_{l'=l+1}^{\infty} \beta_l \beta_{l'} \rho^{l'-l}(l'-l)\bigg|
	&\leq&
	\frac{1}{2}\sum_{l=p+1}^{\infty}\sum_{l'=l+1}^{\infty}\Big( \beta_l^2|\rho|^{l'-l}(l'-l) + \beta_{l'}^2|\rho|^{l'-l}(l'-l) \Big) \\
	&=& \frac{|\rho|}{(1-|\rho|)^2} \sum_{l=p+1}^{\infty} \beta_l^2 \\
	&&+~~\sum_{l'=p+2}^{\infty}\beta_{l'}^2 \frac{|\rho|^{l'-p}}{(1-|\rho|)^2}
	\big(-(l'-p)(1 - |\rho|)-|\rho|\big) \\
	&&+\ \sum_{l'=p+2}^{\infty}\beta_{l'}^2 \frac{|\rho|}{(1-|\rho|)^2} \\
	&\leq & \frac{2|\rho|}{(1-|\rho|)^2} \sum_{l=p+1}^{\infty} \beta_l^2-\sum_{l'=p+1}^{\infty}\beta_{l'}^2 \frac{|\rho|^{l'-p}(l'-p)}{1-|\rho|} \\
	&&-~~ \sum_{l'=p+1}^{\infty}\beta_{l'}^2 \frac{|\rho|^{l'-p}}{1-|\rho|} \\
	&\leq& C \sum_{l=p+1}^{\infty}\beta_{l}^2.
\end{eqnarray*}
For (iii), see that
\begin{eqnarray*}
	\bigg|\sum_{l=1}^{p} \sum_{l'=p+1}^{\infty }\beta_{l} \beta_{l'} \rho^{l'-l}\bigg| &\leq&  \sum_{l=1}^{p} \sum_{l'=p+1}^{\infty }\left(\beta_{l}^{2}|\rho|^{l'-l} + \beta_{l'}^{2}|\rho|^{l'-l} \right) \\
	&=& \sum_{l=1}^{p}\beta_{l}^{2} \sum_{l'=p+1}^{\infty }|\rho|^{l'-l} + \sum_{l'=p+1}^{\infty }\sum_{l=1}^{p}  \beta_{l'}^{2}|\rho|^{l'-l} \\
	&=& \sum_{l=1}^{p}\beta_{l}^2\frac{|\rho|^{p-l+1}}{1-|\rho|} + \sum_{l'=p+1}^{\infty} \beta_{l'}^{2}\bigg(\frac{|\rho|^{l'-p}}{1-|\rho|}-\frac{|\rho|^{l'}}{1-|\rho|}\bigg)\\
	&\le& o(1)+C \sum_{l=p+1}^{\infty} \beta_{l}^{2},
\end{eqnarray*}
where the last inequality follows from \eqref{a3}.

The proof of (iv) follow the same steps as that of (iii).
\end{proof}

\subsection{Proof of Lemma \ref{lemma:lambda_3}}
\label{sec:supp_lambda_3}
\begin{proof}
First, note, that by \eqref{skapa}, Lemma \ref{lemma:alt_expressions} and Remark \ref{remark1}, it holds that
\begin{eqnarray}
	\label{eq:lambda_3_2}
	\lim_{p \to \infty} \sum_{k=1}^{p} (\theta_{k}^{(p)})^2 &=& c_2\  <\   \infty.
\end{eqnarray}
Then, we have that
\begin{align*}
	& \sum_{i,j,k=1}^{p}\big(\rho^{|i-j|} + \theta_{i}^{(p)} \theta_{j}^{(p)}  \big)\big(\rho^{|i-k|} + \theta_{i}^{(p)} \theta_{k}^{(p)}\big)\big(\rho^{|k-j|} + \theta_{k}^{(p)} \theta_{j}^{(p)} \big) \\
	&~~~= \sum_{i,j,k=1}^{p}\big(\rho^{|i-j|+|i-k|} + \rho^{|i-j|} \theta_{i}^{(p)}\theta_{k}^{(p)} ~+~ \rho^{|i-k|} \theta_{i}^{(p)}\theta_{j}^{(p)} + (\theta_{i}^{(p)})^2\theta_{j}^{(p)}\theta_{k}^{(p)}\big) \big(\rho^{|k-j|} + \theta_{k}^{(p)} \theta_{j}^{(p)}  \big) \\
	&~~~= \sum_{i,j,k=1}^{p}\big( \rho^{|i-j|+|i-k|+|k-j|} ~+~ \rho^{|i-j|+|k-j|} \theta_{i}^{(p)}\theta_{k}^{(p)} ~+~  \rho^{|i-k|+|k-j|} \theta_{i}^{(p)}\theta_{j}^{(p)} ~+~ \rho^{|k-j|} (\theta_{i}^{(p)})^2\theta_{j}^{(p)}\theta_{k}^{(p)} \\ & ~~~~~~+~~ \rho^{|i-j|+|i-k|}\theta_{k}^{(p)} \theta_{j}^{(p)} ~+~ \rho^{|i-j|} \theta_{i}^{(p)}(\theta_{k}^{(p)})^2\theta_{j}^{(p)} ~+~\rho^{|i-k|} \theta_{i}^{(p)}(\theta_{j}^{(p)})^2\theta_{k}^{(p)} ~+~ (\theta_{i}^{(p)})^2(\theta_{j}^{(p)})^2(\theta_{k}^{(p)})^2   \big).
\end{align*}
In order to show \eqref{eq:lambda_3_3}, it suffices to show that the following results hold, since the remaining cases will be symmetric:
\begin{enumerate}[label={(\rm \roman*)}]
	\item $\begin{aligned}[t]
		\bigg|\sum_{i,j,k=1}^{p} \rho^{|i-j|+|i-k|+|k-j|} \bigg|~~=~~ o(p^{3/2});
	\end{aligned}$
	\item $\begin{aligned}[t]
		\bigg|\sum_{i,j,k=1}^{p}  \rho^{|i-j|+|k-j|} \theta_{i}^{(p)}\theta_{k}^{(p)} \bigg|~~=~~ o(p^{3/2});
	\end{aligned}$
	\item $\begin{aligned}[t]
		\sum_{i,j,k=1}^{p} \rho^{|k-j|} (\theta_{i}^{(p)})^2\theta_{j}^{(p)}\theta_{k}^{(p)} ~~=~~ o(p^{3/2});
	\end{aligned}$
	\item $\begin{aligned}[t]
		\sum_{i,j,k=1}^{p} (\theta_{i}^{(p)})^2(\theta_{j}^{(p)})^2(\theta_{k}^{(p)})^2 ~~=~~ o(p^{3/2}).
	\end{aligned}$
\end{enumerate}

\noindent \underline{Case (i)}. We have
\begin{eqnarray*}
	\sum_{i,j,k=1}^{p} \rho^{|i-j|+|i-k|+|k-j|} &=& \sum_{i=1}^{p} \sum_{k=1}^{p} \rho^{2|i-k|} ~+~ 2 \sum_{i > j} \sum_{k=1}^{p}\rho^{|i-j|+|i-k|+|k-j|},
\end{eqnarray*}
where
\begin{eqnarray} \label{eq:lambda_3_1}
	\sum_{i=1}^{p} \sum_{k=1}^{p} \rho^{2|i-k|} &=& \sum_{i=1}^{p} \bigg( \sum_{k=1}^{i} \rho^{2(i-k)} + \sum_{k=i+1}^{p} \rho^{2(k-i)} \bigg) ~~=~~ \mathcal{O}(p),
\end{eqnarray}
and
\begin{eqnarray*}
	\sum_{j = 1}^{p} \sum_{i = j+1}^{p} \sum_{k=1}^{p}\rho^{|i-j|+|i-k|+|k-j|} 
	&=& \sum_{j = 1}^{p} \sum_{i = j+1}^{p} \sum_{k=1}^{j}\rho^{2i-2k} ~+~\sum_{j = 1}^{p} \sum_{i = j+1}^{p} \sum_{k=j+1}^{i}\rho^{2i-2j} \\
	&& + ~~  \sum_{j = 1}^{p} \sum_{i = j+1}^{p} \sum_{k=i+1}^{p}\rho^{2k-2j} \\
	&=& \frac{\rho^2}{(1-\rho^2)^3}\Big(p (1-\rho^2) (3 \rho^{2 p}+\rho^2+2) - 2  (1+2 \rho^2) (1-\rho^{2 p}) \Big)
\end{eqnarray*}
so that
\begin{eqnarray*}
	\bigg|\sum_{j = 1}^{p} \sum_{i = j+1}^{p} \sum_{k=1}^{p}\rho^{|i-j|+|i-k|+|k-j|}\bigg| &=& {\cal O}(p) \ =\  o(p^{3/2}).
\end{eqnarray*}

\noindent \underline{Case (ii)}. We have,
\begin{eqnarray*}
	\sum_{i,j,k=1}^{p}  \rho^{|i-j|+|k-j|} \theta_{i}^{(p)}\theta_{k}^{(p)} &\leq&  \sum_{i,j,k=1}^{p} \left(\rho^{2|i-j|+2|k-j|} ~+~ (\theta_{i}^{(p)})^2 (\theta_{k}^{(p)})^2 \right) \\
	&=& \sum_{i,j,k=1}^{p} \rho^{2|i-j|+2|k-j|} ~+~ p \sum_{i,k=1}^{p} (\theta_{i}^{(p)})^2 (\theta_{k}^{(p)})^2.
\end{eqnarray*}
Observe, that by \eqref{eq:lambda_3_2}, we have $ p \sum_{i,k=1}^{p} (\theta_{i}^{(p)})^2 (\theta_{k}^{(p)})^2 = \mathcal{O}(p)$. Additionally,
\begin{eqnarray*}
	\sum_{i,j,k=1}^{p} \rho^{2|i-j|+2|k-j|} &=& \sum_{i,k=1}^{p} \rho^{2|k-i|} ~+~ 2\sum_{j=1}^{p}\sum_{i=j+1}^{p}\sum_{k=1}^{p} \rho^{2(i-j)+2|k-j|}.
\end{eqnarray*}
We use \eqref{eq:lambda_3_1} and note that
\begin{eqnarray*}
	\sum_{j=1}^{p}\sum_{i=j+1}^{p}\sum_{k=1}^{p} \rho^{2(i-j)+2|k-j|}  &=& 	\sum_{j=1}^{p}\sum_{i=j+1}^{p}\sum_{k=1}^{j} \rho^{2(i-j)+2(j-k)} ~+~	\sum_{j=1}^{p}\sum_{i=j+1}^{p}\sum_{k=j+1}^{p} \rho^{2(i-j)+2(k-j)} \\
	&=& \frac{\rho^2}{(1-\rho^2)^3}\Big( p(1-\rho^2) (\rho^{2 p}+\rho^2+1) \\
	&&- (1 - \rho^{2 p})(3\rho^2 + \frac{1-2 \rho^{2 p+2}}{1+\rho^2})\Big) \\
	&=& {\cal O} (p).
\end{eqnarray*}
Thus, it follows that
\begin{eqnarray*}
	\sum_{i,j,k=1}^{p}  \rho^{|i-j|+|k-j|} \theta_{i}^{(p)}\theta_{k}^{(p)} &=& \mathcal{O}(p) \ =\ o(p^{3/2}).
\end{eqnarray*}

\noindent \underline{Case (iii)}.
By \eqref{eq:lambda_3_1} and \eqref{eq:lambda_3_2},
\begin{eqnarray*}
	\sum_{i,j,k=1}^{p}  (\theta_{i}^{(p)})^2  \rho^{|k-j|} \theta_{j}^{(p)}\theta_{k}^{(p)} &=& \sum_{i=1}^{p} (\theta_{i}^{(p)})^2  \sum_{k,j=1}^{p}\rho^{|k-j|} \theta_{j}^{(p)}\theta_{k}^{(p)} \\
	&\leq& \sum_{i=1}^{p} (\theta_{i}^{(p)})^2 \bigg( \sum_{k,j=1}^{p} \Big( \rho^{2|k-j|} + (\theta_{j}^{(p)})^2 (\theta_{k}^{(p)})^2  \Big) \bigg) \\
	&=& o(p^{3/2}).
\end{eqnarray*}

\noindent \underline{Case (iv)}. By \eqref{eq:lambda_3_2},
\begin{eqnarray*}
	\sum_{i,j,k=1}^{p} (\theta_{i}^{(p)})^2(\theta_{j}^{(p)})^2(\theta_{k}^{(p)})^2 &=& \bigg( \sum_{i=1}^{p} (\theta_{i}^{(p)})^2 \bigg)^{3} ~~=~~ o(p^{3/2}).
\end{eqnarray*}

\noindent Thus, this concludes the proof of \eqref{eq:lambda_3_3}.
\end{proof}

\subsection{Proof of result \eqref{eq:supp_1} of Lemma \ref{lemma:alt_expressions}(ii)}
\label{sec:supp_a2}
\begin{proof}
Denote
\begin{eqnarray*}
	K_1(l) &:=& \sum_{k=1}^{p} \rho^{2|k-l|}, \ 	\		 K_{2}(l,l') \ :=\  \sum_{k=1}^{p}\rho^{|k-l|+|k-l'|}.
\end{eqnarray*}
Then, we can write
\begin{eqnarray}
	\sum_{k=1}^{p}\bigg( \sum_{l=1}^{p}\beta_l \rho^{|k-l|}\bigg)^2 &=& \sum_{l=1}^{p}\beta_l^2 K_1(l)+ 2 \sum_{l'>l}\beta_{l} \beta_{l'} K_2(l,l').
\end{eqnarray}
First, note that
\begin{eqnarray*}
	K_1(l) 
	&=&  \frac{-\rho^{-2 l+2 p+2}-\rho^{2 l}+\rho^2+1}{1-\rho^2} ~~\to~~ \frac{-\rho^{2 l}+\rho^2+1}{1-\rho^2} ~~\text{ as } p \to \infty.
\end{eqnarray*}
Then, using the notation by Definition \ref{def:main}, it follows by the Dominated Convergence Theorem (DCT) that
\begin{eqnarray*}
	\sum_{l=1}^{p}\beta_l^2 K_1(l) &\to& \sum_{l=1}^{\infty}\beta_l^2 \frac{1+\rho^2-\rho^{2 l}}{1-\rho^2} ~~=~~ \beta(1) \frac{1+\rho^2}{1-\rho^2} ~-~\beta(\rho^2)\frac{1}{1-\rho^2}.
\end{eqnarray*}
Next, consider $K_2(l,l')$, $l' > l$. It's straightforward to see that,
\begin{eqnarray*}
	\sum_{k=1}^{l} \rho^{|k-l|+|k-l'|} &=& \sum_{k=1}^{l} \rho^{l-k + l' - k } \ =\ \frac{(1-\rho^{2 l})				\rho^{l'-l}}{1-\rho^2}, \\
	\sum_{k=l+1}^{l'} \rho^{|k-l|+|k-l'|} &=& \sum_{k=l+1}^{l'} \rho^{k-l + l' - k } \ =\  (l'-l) \rho^{l'-l}, \\
	\sum_{k=l'+1}^{p} \rho^{|k-l|+|k-l'|} &=& \sum_{k=l'+1}^{p} \rho^{k-l + k-l' } \ =\ \frac{\rho^{-l-l'+2} (\rho^{2 l'}-\rho^{2p})}{1-\rho^2}.
\end{eqnarray*}
By simplifying, it follows that
\begin{eqnarray*}
	K_2(l,l') &=& \frac{(l-l'-1) \rho^{l'-l}+\rho^{
			l+l'}+(-l+l'-1) \rho^{ l'+2-l}+\rho^{2
			p+2-l-l'}}{\rho^2-1} \\
	&\to&  \frac{\rho^{l'-l} ((l'-l)(\rho^2-1) - 1 - \rho^2) +
		\rho^{					l+l'}}{\rho^2-1} ~~\text{ as } p \to \infty,
\end{eqnarray*}
therefore, using the notation of Definition \ref{def:main}, we rewrite
\begin{eqnarray*}
	\sum_{l'=2}^{\infty}\sum_{l=1}^{l'-1}K_2(l,l') &=& \sum_{l'=2}^{\infty}\sum_{l=1}^{l'-1}\beta_{l} \beta_{l'} \Bigg( \rho^{l'-l}(l'-l) + \rho^{l'-l}\frac{1+\rho^2}{1-\rho^2} -  \rho^{l'+l}\frac{1}{1-\rho^2} \Bigg) \\
	&=& b_{1}^{(1)}(\rho) + b_{1}(\rho)\frac{1+\rho^2}{1-\rho^2}- b_{2}(\rho)\frac{1}{1-\rho^2},
\end{eqnarray*}
which concludes the proof of \eqref{eq:supp_1}.
\end{proof}

\subsection{Proof of results \eqref{eq:supp_j1}--\eqref{eq:supp_j2} of Lemma \ref{lemma:alt_expressions}(iii)}
\label{sec:supp_a3}

\begin{proof}
	
	\noindent First, we establish the following observation:
	\begin{eqnarray}
		\sum_{l_2=1}^p|\rho|^{|l_1-l_2|} &=& |\rho|^{l_1} \sum_{l_2=1}^{l_1}|\rho|^{-l_2} + |\rho|^{-l_1} \sum_{l_2=l_1+1}^p|\rho|^{l_2} \nonumber\\
		&=& |\rho|^{l_1} \frac{|\rho|^{-1}(|\rho|^{-l_1}-1)}{|\rho|^{-1}-1}
		+ |\rho|^{-l_1} \frac{|\rho|^{l_1+1}(|\rho|^{p-l_1}-1)}{|\rho|-1}\nonumber\\
		&=& \frac{1+|\rho| -|\rho|^{l_1} - |\rho|^{p-l_1+1}}{1-|\rho|} \  \le\ \frac{1+|\rho|}{1-|\rho|}. \label{uii}
	\end{eqnarray}
	Consider $J_{1}(l)$ in \eqref{J1l}. By (\ref{uii}), write,
	\begin{eqnarray*}
		J_{1}(l) &=&  \sum_{k,k'=1}^p\rho^{|k-k'|} \rho^{|k-l|} \rho^{|k'-l|} {\bf 1}_{\{l=l'\}}\ =\
		\sum_{k=1}^p
		\rho^{2|k-l|} +  2  \sum_{k<k'} \rho^{|k'-k|+|k-l|+|k'-l|}\\
		&=& \frac{1}{1-\rho^2} \left(1+\rho^2-\rho^{2l}-\rho^{2(p-l+1)}\right)+  2  \sum_{k<k'}
		\rho^{|k'-k|+|k-l|+|k'-l|}.
	\end{eqnarray*}
	Observe, that
	\begin{eqnarray}
		\sum_{l=1}^p \beta_l^2 \rho^{2(p-l+1)} &=& \sum_{l=1}^{\lfloor\sqrt p\rfloor}  \beta_l^2 \rho^{2(p-l+1)} ~+~\sum_{l=\lfloor\sqrt p\rfloor+1}^p  \beta_l^2 \rho^{2(p-l+1)} \nonumber \\
		&\le & \rho^{2(p-\lfloor\sqrt p\rfloor+1)}\sum_{l=1}^{\lfloor\sqrt p\rfloor}  \beta_l^2
		~+~ \rho^2 \sum_{l=\lfloor\sqrt p\rfloor+1}^p  \beta_l^2 \ \to \ 0. \label{a3}
	\end{eqnarray}
	Hence,
	\begin{eqnarray}\label{u5}
		\sum_{l=1}^p \beta_l^2 \sum_{k=1}^p
		\rho^{2|k-l|} &\to& \frac{1+\rho^2}{1-\rho^2} \sum_{k=1}^\infty\beta_l^2 ~-~\frac{1}{1-\rho^2} \sum_{k=1}^\infty\beta_l^2\rho^{2l} \notag \\
		&=& \beta(1) \frac{1 + \rho^2}{1 - \rho^2} ~-~  \beta(\rho^2)\frac{1}{1-\rho^2}.
	\end{eqnarray}
	Similarly,
	\begin{eqnarray*}
		\sum_{k<k'}
		\rho^{|k'-k|+|k-l|+|k'-l|} &=& \sum_{k<k'}
		\rho^{2l - 2k }\mathds{1}_{\{l \geq k\}} ~~+~~ \sum_{k<k'}
		\rho^{2k'-2k}\mathds{1}_{\{k < l \leq k' \}} ~~+~~ \sum_{k<k'}
		\rho^{2k'-2l}\mathds{1}_{\{k' > l \}}.
	\end{eqnarray*}
	The first term can be rewritten,
	\begin{eqnarray*}
		\sum_{k<k'}
		\rho^{2l - 2k }\mathds{1}_{\{l \geq k\}} ~~= ~~	\sum_{k=1}^{l} \sum_{k' = k+1}^{l} \rho^{2l-2k}
		&=& \frac{ \rho^{2l}\left(-l + l\rho^2 - \rho^2 \right) + \rho^2}{\left(1-\rho^2\right)^2}.
	\end{eqnarray*}
	We get
	\begin{eqnarray}
		\sum_{l=1}^{p}\beta_{l}^2 \sum_{k<k'}
		\rho^{2l - 2k } &=&\sum_{l=1}^{p}\beta_{l}^2 \frac{ \rho^{2l}\left(-l + l\rho^2 - \rho^2 \right) + \rho^2}{\left(1-\rho^2\right)^2} \notag \\
		&\to& \frac{\rho^2}{(1-\rho^2)^2}\beta(1) ~-~ \frac{1}{1 - \rho^2} \beta^{(1)}(\rho^2) ~-~ \frac{\rho^2}{(1-\rho^2)^2}\beta(\rho^2).~~~~~~~~~ \label{eq:beta_1}
	\end{eqnarray}
	Similarly,
	\begin{eqnarray*}
		\sum_{k<k'}
		\rho^{2k'-2k}\mathds{1}_{\{k < l \leq k' \}} &=& \sum_{k=1}^{l} \sum_{k' = l+1}^{p	} \rho^{2k'-2k} \\&=&-\frac{\rho^{2(p-l+1)}}{\left(1-\rho^2\right)^2}~-~\frac{\rho^{2(l+1)}}{\left(1-\rho^2\right)^2}
		~+~\frac{\rho^{2(p+1)}}{\left(1-\rho^2\right)^2}~+~\frac{\rho^2}{\left(1-\rho^2\right)^2},
	\end{eqnarray*}
	where due to \eqref{a3},
	\begin{eqnarray}
		\sum_{l=1}^{p}\beta_{l}^2 \sum_{k<k'}
		\rho^{2k'-2k}\mathds{1}_{\{k < l \leq k' \}} &\to& \sum_{l=1}^{\infty}\beta_{l}^2\bigg( \frac{\rho^2}{\left(1-\rho^2\right)^2 } -\frac{\rho^{2(l+1)}}{\left(1-\rho^2\right)^2} \bigg) \notag \\
		&=& \frac{\rho^2}{(1-\rho^2)^2}\beta(1) ~-~ \frac{\rho^2 }{(1-\rho^2)^2}\beta(\rho^2).~~~ \label{eq:beta_2}
	\end{eqnarray}
	Finally, observe that
	\begin{eqnarray*}
		\sum_{k<k'}
		\rho^{2k'-2l}\mathds{1}_{\{k' < l \}}   &=& \sum_{k=l+1}^{p-1} \sum_{k' = k+1}^{p	} \rho^{2k'-2l} \\
		&=& \frac{(p-l) \rho^{2(p-l+1)}}{\left(1-\rho^2\right)^2} ~~-~~\frac{(p-l+1) \rho^{2(p-l+2)}}{\left(1-\rho^2\right)^2}
		~~+~~ \frac{\rho^4}{\left(1-\rho^2\right)^2},
	\end{eqnarray*}
	where due to \eqref{a3}, it remains to see that, as $p \to \infty$,
	\begin{eqnarray}
		\sum_{l=1}^{p}\beta_{l}^2 \sum_{k<k'}
		\rho^{2k'-2l}\mathds{1}_{\{k' < l \}}  &\to& \sum_{l=1}^{\infty}\beta_l^2 \frac{\rho^4}{(1-\rho^2)^2} ~=~ \frac{\rho^4}{(1-\rho^2)^2} \beta(1). \label{eq:beta_3}
	\end{eqnarray}
	Finally, by collecting the  terms of (\ref{eq:beta_1}), (\ref{eq:beta_2}) and (\ref{eq:beta_3}) and simplifying, we get
	\begin{eqnarray*}
		\lim_{p\to\infty}\sum_{l=1}^{p}\beta_{l}^2 \sum_{k<k'}^{p}
		\rho^{|k'-k|+|k-l|+|k'-l|} &=& \frac{\rho^4}{(1-\rho^2)^2} \beta(1) -\frac{\rho^2 }{(1-\rho^2)^2}\beta(\rho^2) ~+~ \frac{2\rho^2}{(1-\rho^2)^2}\beta(1)
		\\
		&&-~~ \frac{1}{1 - \rho^2} \beta^{(1)}(\rho^2) ~-~ \frac{\rho^2}{(1-\rho^2)^2}\beta(\rho^2) \\
		&=& \beta(1)\Big( \frac{\rho^2(\rho^2 + 2)}{(1-\rho^2)^2} \Big) - \beta(\rho^2)\Big(\frac{2\rho^2}{(1-\rho^2)^2} \Big) ~-~ \frac{ \beta^{(1)}(\rho^2) }{1 - \rho^2} .
	\end{eqnarray*}
	Therefore, from \eqref{u5}, (\ref{eq:beta_1}), (\ref{eq:beta_2}) and (\ref{eq:beta_3}),
	\begin{eqnarray*}
		\sum_{l=1}^{p}\beta_{l}^{2} J_1(l) &\to& \beta(1) \frac{1 + \rho^2}{1 - \rho^2} ~-~ \frac{1}{1-\rho^2} \beta(\rho^2) \\
		&+& 2\bigg( \beta(1) \frac{\rho^2(\rho^2 + 2)}{(1-\rho^2)^2} ~-~ \beta(\rho^2)\,\frac{2\rho^2}{(1-\rho^2)^2} ~-~ \frac{1}{1 - \rho^2} \beta^{(1)}(\rho^2)  \bigg) \\
		&=&\beta(1)\, \frac{\rho^4+4 \rho^2+1}{(1-\rho^2)^2} ~-~ \beta(\rho^2)\,\frac{3 \rho^2+1}{(1-\rho^2)^2} ~-~ \frac{2}{1 - \rho^2} \beta^{(1)}(\rho^2),
	\end{eqnarray*}
	which concludes the proof of \eqref{eq:supp_j1}.\smallskip
	
	Next, consider $J_{2}(l,l')$ in \eqref{J2ll}. According to the arrangement of indices $k,k,l,l'$, we have 9 cases:  		
	\begin{enumerate}
		\item $k,k'\in\{1,\dots,l\}$,
		\item $k\le l$, $l<k\le l'$,
		\item $k\le l$, $l'<k'\le p$,
		\item $l< k\le l'$, $1\le k'\le l$,
		\item $k,k'\in \{l+1,\dots,l'\}$,
		\item $l< k\le l'$, $l'< k'\le p$,
		\item $l'< k\le p$, $1\le k'\le l$,
		\item $l'< k\le p$, $l< k'\le l'$,
		\item $ k,k'\in \{l',\dots,p\}$.
	\end{enumerate}

	\noindent \underline{Case 1}:
	\begin{eqnarray*}
		\sum_{k,k'=1}^p\rho^{|k-k'|} \rho^{|k-l|} \rho^{|k'-l'|}{\bf 1}_{\{k,k'\le l\}}
		&=& \rho^{l+l'}\sum_{k,k'=1}^l\rho^{|k-k'|-k-k'} \\
		&=& \frac{(-(2 l+1) \rho^{2 l}+(2l-1) \rho^{2 l+2}+\rho^2+1)
			\rho^{l'-l}}{(1-\rho^2)^2}.
\end{eqnarray*}

\noindent \underline{Case 2}:
\begin{eqnarray*}
	\sum_{k,k'=1}^p\rho^{|k-k'|} \rho^{|k-l|} \rho^{|k'-l'|}{\bf 1}_{\{k\le l, l<k'\le l'\}}
	&=& \rho^{l+l'}\sum_{k=1}^l\sum_{k'=l+1}^{l'} \rho^{-2k} \\
	&=& \frac{(1-\rho^{2 l})
		(l'-l) \rho^{l'-l}}{1-\rho^2}.
\end{eqnarray*}
\noindent \underline{Case 3}:
\begin{eqnarray*}
	\sum_{k,k'=1}^p\rho^{|k-k'|} \rho^{|k-l|} \rho^{|k'-l'|}{\bf 1}_{\{k\le l, l'<k'\le p\}}
	&=& \rho^{l-l'}\sum_{k=1}^l\sum_{k'=l'+1}^p \rho^{-2k+2k'} \\
	&=& \frac{(1-\rho^{2 l}) \rho^{-l'-l+2}
		(\rho^{2l'}-\rho^{2p})}{(1-\rho^2)^2}.
\end{eqnarray*}
\noindent \underline{Case 4}:
\begin{eqnarray*}
	\sum_{k,k'=1}^p\rho^{|k-k'|} \rho^{|k-l|} \rho^{|k'-l'|}{\bf 1}_{\{l<k\le l', 1\le k'\le l\}}
	&=& \rho^{-l+l'}\sum_{k=l+1}^{l'}\sum_{k'=1}^l \rho^{2k-2k'} \\
	&=&\frac{(1-\rho^{2 l}) \rho^{l'-3 l+2}
		(\rho^{2 l}-\rho^{2
			l'})}{(1-\rho^2)^2}.
\end{eqnarray*}
\noindent \underline{Case 5}:
\begin{eqnarray*}
	\sum_{k,k'=1}^p\rho^{|k-k'|} \rho^{|k-l|} \rho^{|k'-l'|}{\bf 1}_{\{l<k,k'\le l'\}}
	&=& \rho^{-l+l'}\sum_{k,k'=l+1}^{l'}\rho^{|k-k'|+k-k'} \\
	&=&\rho^{-l+l'}\bigg(\sum_{k=l+1}^{l'}\sum_{k': k' \geq k}^{l'}1 + \sum_{k=l+1}^{l'} \sum_{k': k' \geq l+1}^{k-1}\rho^{2k-2k'} \bigg) \\
	&=& \frac{1}{2} \rho^{l'-3 l} \bigg(\rho^{2 l}
	(-l'+l-1)		(l-l')\\
	&&+~~\frac{2 \rho^2	\big((1-\rho^2) \rho^{2 l} l'+\rho^{2
			l'}-(l+1) \rho^{2 l}+l \rho^{2
			l+2}\big)}{(1-\rho^2)^2}\bigg).
\end{eqnarray*}

\noindent \underline{Case 6}:

\begin{eqnarray*}
	\sum_{k,k'=1}^p\rho^{|k-k'|} \rho^{|k-l|} \rho^{|k'-l'|}{\bf 1}_{\{l<k\le l', l'<k'\le p\}}
	&=& \rho^{-l-l'}\sum_{k=l+1}^{l'}\sum_{k'=l'+1}^p \rho^{2k'} \\
	&=& \frac{\left(l'-l\right) \rho^{-l'-l+2}
		(\rho^{2 l'} - \rho^{2 p})}{1-\rho^2}. 
\end{eqnarray*}
\noindent \underline{Case 7}:
\begin{eqnarray*}
	\sum_{k,k'=1}^p\rho^{|k-k'|} \rho^{|k-l|} \rho^{|k'-l'|}{\bf 1}_{\{l'<k\le p, 1\le k'\le l\}}
	&=& \rho^{-l+l'}\sum_{k=l'+1}^p \sum_{k'=1}^l \rho^{2k-2k'} \\
	&=& \frac{(1-\rho^{2 l}) \rho^{l'-3 l+2}
		(\rho^{2 l'}-\rho^{2				p})}{(1-\rho^2)^2}.
\end{eqnarray*}
\noindent \underline{Case 8}:
We have
\begin{eqnarray*}
	\sum_{k,k'=1}^p\rho^{|k-k'|} \rho^{|k-l|} \rho^{|k'-l'|}{\bf 1}_{\{l'<k\le p, l< k'\le l'\}}
	&=& \rho^{-l+l'}\sum_{k=l'+1}^p \sum_{k'=l+1}^{l'}\rho^{2k-2k'} \\
	&=& \frac{\rho^{-3 l-l'+2} (\rho^{2 l}-\rho^{2 l'})(\rho^{2
			l'}-\rho^{2 p})}{(1-\rho^2)^2}.
\end{eqnarray*}

\noindent \underline{Case 9}:
We have
\begin{eqnarray*}
	\sum_{k,k'=1}^p\rho^{|k-k'|} \rho^{|k-l|} \rho^{|k'-l'|}{\bf 1}_{\{l'<k,k'\le p\}}
	&=& \rho^{-l-l'}\sum_{k,k'=l'+1}^p\rho^{|k-k'|+k+k'}  \\
	&=& \rho^{-l-l'}\bigg( \sum_{k=l'+1}^p \sum_{k'=l'+1 }^{k} \rho^{2k}
	+ \sum_{k=l'+1}^p \sum_{k'=k+1}^p \rho^{2k'} \bigg)  \\
	&=& \frac{\rho^{-l'-l+2}}{(1-\rho^2)^2}\big(2 (1-\rho^2)
	l' \rho^{2 p}+\rho^{2 l'}+\rho^{2 l'+2}\\&&-~(2
	p+1) \rho^{2 p}+(2 p-1) \rho^{2p+2}\big).
\end{eqnarray*}

\noindent Observe, that
\begin{eqnarray}
	\bigg|	\sum_{l'=2}^{p}\sum_{l=1}^{l'-1}\beta_l \beta_{l'}\rho^{2p -l -l'}(l+l'-2p) \bigg| &\leq&	\bigg|\sum_{l'=2}^{p}\sum_{l=1}^{l'-1}\beta_l \beta_{l'}\rho^{p-l} \rho^{p-l'}(p-l)\bigg|  \notag  \\
	&&+	\bigg|\sum_{l'=2}^{p}\sum_{l=1}^{l'-1}\beta_l \beta_{l'}\rho^{p-l} \rho^{p-l'}(p-l')\bigg|. \label{eq:s15}
\end{eqnarray}
The two summands of \eqref{eq:s15} are symmetric, therefore due to brevity we consider only the first term. The proof for the second term will be analogous. Note,
\begin{eqnarray} \label{eq:s16}
	\bigg|\sum_{l'=2}^{p}\sum_{l=1}^{l'-1}\beta_l \beta_{l'}\rho^{p-l} \rho^{p-l'}(p-l)\bigg| &\leq & \sum_{l'=2}^{p}|\beta_{l'}| |\rho|^{p-l'}\sum_{l=1}^{p-1}|\beta_l| |\rho|^{p-l}(p-l),
\end{eqnarray}
where
\begin{eqnarray*}
	\sum_{l=1}^{p-1}|\beta_l| |\rho|^{p-l}(p-l) &\leq& \bigg(\sum_{l=1}^{p-1}\beta_l^2\bigg)^{1/2} \bigg(\sum_{l=1}^{p-1} \rho^{2(p-l)}(p-l)^2\bigg)^{1/2} < \infty
\end{eqnarray*}
holds due to  $\sum_{j=1}^{\infty}\beta_j^2 < \infty$. It remains to note that
\begin{eqnarray*}
	\sum_{l'=2}^{p}|\beta_{l'}| |\rho|^{p-l'} &=& \sum_{l'=2}^{\lfloor \sqrt{p}\rfloor }|\beta_{l'}| |\rho|^{p-l'} + \sum_{l'=\lfloor \sqrt{p}\rfloor  + 1}^{p }|\beta_{l'}| |\rho|^{p-l'} \\
	&\leq& |\rho|^{p - \lfloor \sqrt p \rfloor } p^{1/4} \bigg(\sum_{l'=2}^{\lfloor \sqrt{p} \rfloor }\beta_{l'}^2\bigg)^{1/2} + \bigg( \sum_{l'=\lfloor \sqrt{p}\rfloor  + 1}^{p }\beta_{l'}^2 \bigg)^{1/2} \bigg( \sum_{l'=\lfloor \sqrt{p}\rfloor  + 1}^{p }\rho^{2(p-l')}\bigg)^{1/2} \\
	&\to& 0,
\end{eqnarray*}
since $\sum_{l'=2}^{\lfloor \sqrt{p} \rfloor }\beta_{l'}^2 < \infty$ and $ \sum_{l'=\lfloor \sqrt{p}\rfloor  + 1}^{p }|\rho|^{2(p-l')} < \infty$. Thus, by \eqref{eq:s15} and \eqref{eq:s16}, it follows that
\begin{eqnarray}
	\bigg|	\sum_{l'=2}^{p}\sum_{l=1}^{l'-1}\beta_l \beta_{l'}\rho^{2p -l -l'}(l+l'-2p) \bigg| &\to& 0, \text{ as } p \to \infty.  \label{eq:s17}
\end{eqnarray}

\noindent Due to the results of \eqref{a3} and \eqref{eq:s15}--\eqref{eq:s17}, it follows that the collected Cases 1--9 can be greatly simplified, leading to:
\begin{eqnarray*}
	J_{2}(l,l') &\to&
	\frac{1}{2(1-\rho^2)^2}\big(\rho^{l'-l}((l'-l)^2 (1-\rho^2)^2 + 3(1-\rho^4)(l'-l) + 2(1 + 4\rho^2 + \rho^4))\\
	&&  +\ \rho^{l'+l}(-2(1-\rho^2)(l'+l) - 2 - 6\rho^2)\big).
\end{eqnarray*}

Therefore, by DCT, as $p \to \infty$, we have
\begin{eqnarray*}
	\sum_{l'>l}\beta_{l}\beta_{l'} J_2(l,l') &=&
	\frac{1}{2(1-\rho^2)^2}\bigg( \sum_{l'>l}\beta_{l}\beta_{l'}\rho^{l'-l}
	\big( (l'-l)^2 (1-\rho^2)^2 \\
	&&+\ 3(1-\rho^4)(l'-l) + 2(1 + 4\rho^2 + \rho^4)\big) \\
	&&+\ \sum_{l'>l}\beta_{l}\beta_{l'}\rho^{l'+l} \big(-2(1-\rho^2)(l'+l) - 2 - 6\rho^2  \big) \\
	&\to& \frac{1}{2(1-\rho^2)^2}\big( b^{(2)}(\rho) (1-\rho^2)^2
	+ 3b_{1}^{(1)}(\rho) (1-\rho^4) \\
	&&+~~ 2b_{1}(\rho)(1 + 4\rho^2 + \rho^4) - 2b_{2}^{(1)}(\rho)(1-\rho^2) -2 b_{2}(\rho)(1+3\rho^2) \big),
\end{eqnarray*}
which concludes the proof of \eqref{eq:supp_j2}.
\end{proof}

\end{document}